\documentclass[12pt]{amsart}
\usepackage{amsmath,amsfonts,amssymb,amsthm,amstext,pgf,graphicx,hyperref,verbatim,lmodern,textcomp,color,young,tikz}
\usetikzlibrary{decorations}
\usepackage[mathscr]{euscript}
\usetikzlibrary{decorations.markings}
\usetikzlibrary{arrows}
\usepackage{float}
\theoremstyle{plain}
\newtheorem{theorem}{Theorem}[section]
\newtheorem{lemma}[theorem]{Lemma}
\newtheorem{proposition}[theorem]{Proposition}
\newtheorem{corollary}{Corollary}
\newtheorem{rem}{Remark}

\theoremstyle{definition}
\newtheorem{defn}{Definition}[section]

\title{}
\begin{document}
\title[connectivity and spectral radius of proper enhanced power graphs]{On connectivity, domination number and spectral radius of the proper enhanced power graphs of finite nilpotent groups}
\author[Sudip Bera]{Sudip Bera}
\author[Hiranya Kishore Dey]{Hiranya Kishore Dey}
\address[Sudip Bera]{Department of Mathematics, Indian Institute of Science, Bangalore 560 012}
\email{sudipbera@iisc.ac.in}
\address[Hiranya Kishore Dey]{Department of Mathematics, Indian Institute of Technology, Bombay, India.}
\email{hkdey@math.iitb.ac.in} 
\maketitle	
\begin{abstract}
For a group $G,$ the enhanced power graph of $G$ is a graph with vertex set $G$ in which two distinct elements $x, y$ are adjacent if and only if there exists an element $w$ in $G$ such that both $x$ and $y$ are powers of $w.$ The proper enhanced power graph is the induced subgraph of the enhanced power graph on the set $G \setminus S,$ where $S$ is the set of dominating vertices of the enhanced power graph. In this paper we at first characterize the dominating vertices of enhanced power graph of any finite nilpotent group. Thereafter, we classify all nilpotent groups $G$ such that the proper enhanced power graphs are connected and find out their diameter. We also explicitly find out the domination number of proper enhanced power graphs of finite nilpotent groups. Finally, we determine the multiplicity of the Laplacian spectral radius of the enhanced power graphs of nilpotent groups. 
\end{abstract}
\section{Introduction}
\label{sec:intro} 
The study of graphs associated with various algebraic structures is a topic of increasing interest during the last two decades. The benefit of studying these graphs is
multifold.  They help us to  (1) characterize the resulting graphs, (2) 
characterize the algebraic structures with isomorphic graphs, and (3) also to realize 
the interdependence between the algebraic structures and the corresponding graphs. Besides, these graphs have important applications (see, for example, 
\cite{surveypwrgraphkac1, cayleygraphsckry}) and they are related to automata theory \cite{automatatheory}. Many different types of graphs, specifically
%semiring \cite{atani1}, semigroup \cite{deMeyer}, poset \cite{joshizerodivgraphofideal},
 power graph of semigroup \cite{undpwrgraphofsemgmainsgc1, directedgrphcompropofsemgrpkq3}, group \cite{combinatorialpropertyandpowergraphsofgroupskq1},  intersection power graph of group \cite{intersectionpwegraphb3}, enhanced power graph of a group \cite{firstenhcedpwrstrctreaacns1,enhancedpwrgrapbb3}, comaximal subgroup graph \cite{das-saha-saba} etc. have been introduced to explore the properties of algebraic structures using graph theory.
The concept of a power graph was introduced in the context of semigroup theory by Kelarev and Quin \cite{combinatorialpropertyandpowergraphsofgroupskq1}.
% As
%explained in the survey \cite{surveypwrgraphkac1}, that definition of power graph also covered the undirected graphs. This paper follows Chakrabarty et al. and we use the brief term ``power graph'' which is defined as follows.
\begin{defn}[\cite{surveypwrgraphkac1,undpwrgraphofsemgmainsgc1, combinatorialpropertyandpowergraphsofgroupskq1}]\label{defn: powr graph}
	Given a group $G,$ the \emph{power graph} $\mathcal{P}(G)$ of $G$ is a simple graph with vertex set $G$ and two vertices $u$ and $v$ are connected by an edge if and only if one of them is the power of another.  
\end{defn}

Another well-studied graph, called \emph{commuting graph} associated with a group $G$ has been studied in \cite{braurflower} as a part of the classification of finite simple groups. For more information about the commuting graph, we refer to
 \cite{ijraeljofmathematics, braurflower, onboundingdiamcommuting}.
\begin{defn}[\cite{braurflower}]\label{defn: commuting graph}
	Let $G$ be a group. The \emph{commuting graph} of $G,$ denoted by $\mathcal{C}(G),$ is the simple graph whose vertex set is a set of non-central elements of $G$ and two distinct vertices $u$ and $v$ are joined if and only if $u $ and $v$ commutes, that is, $uv=vu.$ 	
\end{defn}
In this paper, our topic is the enhanced power graph of a group which is introduced by Alipour et al. in \cite{firstenhcedpwrstrctreaacns1} 
 as follows:
\begin{defn}[\cite{firstenhcedpwrstrctreaacns1}]\label{defn: enhcdpowr graph}
Let $G$ be a group. The \emph{enhanced power graph} of $G,$
denoted by $\mathcal{G}_E(G),$ is the graph with vertex set $G,$ in which two vertices $u$ and $v$ are joined if and only
if there exists an element $w \in G$ such that both $u \in \langle w \rangle $ and $v \in  \langle w \rangle.$ 
\end{defn}
Lots of works have been done recently studying various properties of the enhanced power graph of the finite group. The authors in \cite{firstenhcedpwrstrctreaacns1} 
characterized the finite groups such that any arbitrary pair of these three graphs (power, commuting, enhanced) are equal.
Besides, in \cite{Zahirovienhnacedpwrgraph}, the researchers proved that finite groups with isomorphic enhanced power graphs have isomorphic directed power graphs.
Bera et al. in \cite{enhancedpwrgrapbb3} studied the completeness, dominatability and many other interesting properties of the enhanced power graph. 
Ma and She  in \cite{ma-she} derived the metric dimension whereas Hamzeh et al. in \cite{Hamzeh-ashrafi} derived the automorphism groups 
of enhanced power graphs of finite groups.
 In this paper, we study the connectivity, dominatibility, diameter and Laplacian spectral radius of enhanced power graphs of groups.   
\subsection{Basic definitions and notations}\label{subsection:definitions}
For the convenience of the reader and also for later use, we recall some basic
definitions and notations about graphs. 
Let $\Gamma$ be a graph with vertex set $V$. 
Two elements $u$ and $v$ are said to be adjacent if there is an edge between them. For a vertex $u$, we denote by $N(u)$ the set of vertices which are adjacent to $u$. For a set $V_1 \subseteq V$, define $N(V_1)= \displaystyle \cup  _{u \in V_1} N(u).$ A \emph{path} of length $k$ between two vertices $v_0$ and $v_k$ is an alternating
sequence of vertices and edges $v_0, e_0, v_1, e_1, v_2, \cdots , v_{k-1}, e_{k-1}, v_k$, where the $v_i'$s are distinct
(except possibly the first and last vertices) and $e_i'$s are the edges $(v_i, v_{i+1}).$  A graph $\Gamma$ is said to be \emph{connected} if for any pair of vertices $u$ and $v,$ there exists a path between $u$ and $v.$ The distance between two vertices $u$ and $v$ in a connected graph $\Gamma$ is the length of the shortest path between them and it is denoted by $d(u, v).$ Clearly, if $u$ and $v$ are adjacent, then $d(u, v)=1.$ For a graph $\Gamma,$ its \emph{diameter} is defined as $\text{diam}(\Gamma)= \max_{u, v \in V} d(u, v).$ That is, the diameter of graph is the largest possible distance between pair of vertices of a graph. $\Gamma$ is said to be \emph{complete} if any two distinct vertices are adjacent. For a disconnected graph $\Gamma$, we denote the set of connected components of it by $C(\Gamma).$

 A vertex of a graph $\Gamma$ is called a 
\emph{dominating vertex} if it is adjacent to every other
vertex. For a graph $\Gamma,$ let $\text{\emph{Dom}}(\Gamma)$ denote the set of all dominating vertices in $\Gamma.$  The \emph{vertex connectivity} of a graph $\Gamma,$ denoted by $\kappa{(\Gamma)}$ is
the minimum number of vertices that need to be removed from the vertex set $\Gamma$ so that the
induced subgraph of $\Gamma$ on the remaining vertices is disconnected. The complete graph with $n$ vertices has vertex connectivity $n-1.$  A set $S \subseteq V(\Gamma)$ is said to be a \emph{dominating set} if every vertex of $V \setminus S$ is adjacent to some vertex of $S$. The minimum possible number of a dominating set is called the \emph{domination number} and it is denoted by $\gamma(\Gamma).$ From the definition of the enhanced power graph, it is clear that the identity of the group is always a dominating vertex.  
The enhanced power graph is called \emph{dominatable} if it has a dominating vertex other than identity.  For more on graph theory we refer \cite{graphthrybondymurti, algbraphgodsil, graphthrywest}.

Throughout this paper we consider $G$ as a finite group.  $|G|$ denotes the cardinality of the set $G.$ For a prime $p,$ a group $G$ is said to be a $p$-group if $|G|=p^{r}, r\in \mathbb{N}.$ If $|G|= p_{\ell}^{r}$ for some prime $p_{\ell}$, then we say that $G$ is a $p_{\ell}$-group. For a subgroup $H$ of $G$, we call $H$ a $p$-order subgroup if $|H|=p.$ For two $p$-order subgroups $H_1$ and $H_2$, we say $H_1$ and $H_2$ to be distinct if $|H_1 \cap H_2|=1.$  For a $p$-group $G_p$, let $\text{exp}(G_p)$ be the highest natural number $t$ such that there exists an element of order $p^t$ in $G_p$.  
 For any element $g \in G, \text{o}(g)$ denotes the order of the element $g .$
%  Let $G$ be a group and $a\in G,$ then $\mathcal{G}en(a)$ is the set of all generators of the cyclic group $\langle a\rangle.$
   Let $m$ and $n$ be any two positive integers, then the greatest common divisor of $m$ and $n$ is denoted by $\text{gcd}(m, n).$ The Euler's phi function $\phi(n)$ is the number of integers $k$ in the range $1 \leq k \leq n$ for which the  $\text{gcd}(n, k)$ is equal to $1.$ The set $\{1, 2, \cdots, n\}$ is denoted by $[n].$
   
The plan of the paper is as follows. In Section \ref{sec:main_results} we state our main results, and in Section \ref{Sec:Preliminaries}, we mention some earlier known results. In Section \ref{sec:dom_vertices}, we completely characterize the dominating vertices of the enhanced power graphs of finite nilpotent groups. In Section \ref{sec:connectivity_diameter}, we study about the connectivity and diameter of proper enhanced power graphs of nilpotent groups. The domination number of proper enhanced power graphs is studied in Section \ref{sec:domination_number_proper}. Finally, multiplicity of the Laplacian spectral radius of enhanced power graphs is given in Section \ref{sec:mul_sec_radius}.
\section{Main results}
\label{sec:main_results}
In this section, we state and motivate our main results of this paper. 
If a non-complete graph $\Gamma$ has a dominating vertex, then clearly the graph is connected, has domination number $1$ and diameter $2$. Therefore, for any graph with a dominating vertex, the properties connectivity, domination number, diameter are not interesting. In this respect, the authors in \cite[Question 40]{firstenhcedpwrstrctreaacns1} asked about the connectivity of power graphs when all the dominating vertices are removed. Recently, Cameron and Jafari in \cite{HeidarJafari} answered this question for power graphs. Bera et al. in \cite{bera-dey-sajal} answered the same question for the enhanced power graphs of finite abelian groups. In this paper, we investigate the connectivity, domination number, and diameter for the enhanced power graphs of finite nilpotent groups after the dominating vertices are removed. To seek the answer to this question, define the following graph:
%\begin{defn}\label{defn: proper pwr graph}
%	Given a group $G,$ the \emph{proper power graph} of $G,$ denoted by $\mathcal{P}^{**}(G),$ is the graph obtained by deleting all the dominating vertices from the power graph $\mathcal{P}(G).$ Moreover, by $\mathcal{P}^{*}(G)$ we denote the graph obtained by deleting only the identity element of $G$ and this is called \emph{deleted power graph} of $G.$ Note that if there is no such dominating vertex other than identity, then $\mathcal{P}^{*}(G)=\mathcal{P}^{**}(G).$
%\end{defn}
\begin{defn}\label{defn: proper enhacd pwr graph}
For a group $G,$ the \emph{proper enhanced power graph} of $G,$ denoted by $\mathcal{G}^{**}_E(G),$ is the graph obtained by deleting all the dominating vertices from the enhanced power graph $\mathcal{G}_E(G).$
% Besides, by $\mathcal{G}^{*}_E(G)$ we denote the graph obtained by deleting only the identity element of $G$ and this is called \emph{deleted enhanced power graph} of $G.$ Thus if there is no such dominating vertex other than identity, then $\mathcal{G}_{E}^{*}(G)=\mathcal{G}_{E}^{**}(G).$
\end{defn}
Therefore, for studying the proper enhanced power graph, we first need to characterize all the dominating vertices of the graph $\mathcal{G}_E(G)$ for a finite group $G$. Cameron and Jafari in \cite{HeidarJafari} characterized the dominating vertices of the power graph for any finite group $G$. Bera et al. in \cite{enhancedpwrgrapbb3} characterized the 
dominatability of the enhanced power graph for any finite abelian group $G.$ In this paper, we at first extend this result to finite nilpotent groups and for that purpose, we first recall the following structure of a finite nilpotent group.  
\subsection{Nilpotent group}\label{Defn: Nilpotent group}
A finite group $G$ is nilpotent if and only if $G\cong P_1\times \cdots\times P_r,$ where for each $i\in[r], P_i$ is a sylow subgroup of order $p_i^{t_i}$ of $G.$ So, for a finite nilpotent group $G,$ we have the following cases:
\begin{enumerate}
\item 
No sylow subgroups of $G$ are either cyclic or generalized quaternion.
\item
$G$ has cyclic sylow subgroups. In this case, $G \cong G_1\times \mathbb{Z}_n,$ where $G_1$ is a nilpotent group having no sylow subgroups which are either cyclic or generalized quaternion and $\text{gcd}(|G_1|, n)=1.$
\item
$G$ has a sylow subgroup isomorphic to generalized quaternion. Here $G \cong G_1\times Q_{2^k},$ where $G_1$ is described as $(2)$ and $\text{gcd}(|G_1|, 2)=1.$
\item
$G$ has sylow subgroups which are cyclic and generalized quaternion. In this case, $G \cong G_1\times \mathbb{Z}_n\times Q_{2^k},$ where $G_1$ is described as $(2)$ and $\text{gcd}(|G_1|, n)=\text{gcd}(|G_1|, 2)=\text{gcd}(n, 2)=1.$ 
\end{enumerate}
We now characterize the dominatability of the enhanced power graph for any finite nilpotent group $G$. 
\begin{theorem}\label{Thm: G' nilpotent dominatability iff}
Let $G$ be a finite nilpotent group. Then $\text{Dom}(\mathcal{G}_E(G))=\{ e \}$ if and only if no sylow subgroups of $G$ are either cyclic or generalized quarternion.
%	\begin{enumerate}
%		\item 
%		$G=G_1\times \mathbb{Z}_n,$ where $G_1$ is a nilpotent group such that $G_1$ has no sylow subgroups which are either cyclic or generalized quaternion and $\text{gcd}(|G_1|, n)=1$
%		\item
%		$G=G_1\times Q_{2^k},$ where $G_1$ is described as $(1)$ and $\text{gcd}(|G_1|, 2)=1$
%		\item
%		$G=G_1\times \mathbb{Z}_n\times Q_{2^k},$ where $G_1$ is described as $(1)$ and $\text{gcd}(|G_1|, n)=\text{gcd}(|G_1|, 2)=\text{gcd}(n, 2)=1.$ 
%	\end{enumerate}	
\end{theorem} 
Subsequently, we characterize all the dominating vertices of the enhanced power graph $\mathcal{G}_E(G)$ for any finite nilpotent group $G$. Theorem \ref{Thm: Dom set for nilpotent grp} tells about that. We move on to the connectivity of the proper enhanced power graph $\mathcal{G}_E^{**}(G)$ and our approach depends on whether $G$ has a sylow subgroup that is generalized quarternion. In this context, we have the following two results which completely characterize the connectivity of the graph $\mathcal{G}_E^{**}(G).$ 
\begin{theorem}\label{Thm:no sylow subgroups quaternion, ** connected iff}
Let $G$ be a finite nilpotent group that does not have any sylow subgroups which are generalized quaternion. That is, 
\begin{enumerate}
\item 
either $G=G_1,$ where $G_1$ is a finite nilpotent group which does not have any sylow subgroups which are either cyclic or generalized quaternion,
\item
or $G=G_1\times \mathbb{Z}_n,$ where $G_1$ is described as $(1)$ and $\text{gcd}(|G_1|, n)=1.$	
\end{enumerate}
Then $\mathcal{G}^{**}_E(G)$ is disconnected if and only if $G_1$ is $p$-group.	
\end{theorem}
%In this case, about the connectivity of the proper enhanced power graph we have the following;  
%\begin{theorem}\label{Thm:no sylow subgroups quaternion, ** connected iff}
%	Let $G$ be a nilpotent group having no sylow subgroups which are generalized quaternion. Then $\mathcal{G}^{**}_E(G)$ is disconnected if and only if $G_1$ is $p$-group.	
%\end{theorem}
%Now we assume that the nilpotent group $G$ has sylow subgroup which is generalized quaternion. In this case $G$ is one of the following; 
%\begin{enumerate}
%	\item 
%	either $G=G_1\times Q_{2^k},$ where $G_1$ is a finite nilpotent group having no sylow subgroups which are either cyclic or generalized quaternion and $\text{gcd}(|G_1|, 2)=1$
%	\item
%	or $G=G_1\times \mathbb{Z}_n\times Q_{2^k},$ where    $\text{gcd}(|G_1|, n)=\text{gcd}(|G_1|, 2)=\text{gcd}(2, n)=1.$	
%\end{enumerate}
%Here the next theorem describes the result about the  connectivity of proper enhanced power graph of the nilpotent group $G$ having a sylow subgroup isomorphic to generalized quaternion group. 

\begin{theorem}\label{Thm: having sylow subgroups quaternion, ** connected }
	Let $G$ be a finite nilpotent group having a sylow subgroup which is generalized quaternion. Then $\mathcal{G}^{**}_E(G)$ is connected.	
\end{theorem}

We also find the diameter of the proper enhanced power graph $\mathcal{G}_E^{**}(G)$ for any finite nilpotent group $G$ in
Theorem \ref{thm:diam_prop_enhcd_pwr_abelian}. Moreover, we improve the earlier known upper bound on the vertex  connectivity of the enhanced power graph of a finite abelian group. In this context, our result is Theorem \ref{improved vc of enhced pwr raph for grnrral abelian grp}.

It is clear that for any finite group $G$, the proper enhanced power graph $\mathcal{G}_E^{**}(G)$ has no dominating vertex and therefore clearly, $\gamma(\mathcal{G}_E^{**}(G))>1.$ One of the main goals of this paper is to find the domination number for the proper enhanced power graph $\gamma(\mathcal{G}_E^{**}(G))$. The way we approach towards finding the domination number depends on the connectivity of the enhanced power graph.  We first consider the case when $\mathcal{G}_E^{**}(G)$ is disconnected and on this theme, we have the following result. 

\begin{theorem}
	\label{thm:domination number p-group times cyclic_nilpotent}
	Let $G_1$ be a finite  $p$-group which is neither cyclic nor generalized quarternion. Then 
	$\gamma(\mathcal{G}^{**}_E(G_1))= \text{number of distinct p-order subgroups of } G_1 .$ 
	Let $G$ be a nilpotent group such that  $G = G_1 \times \mathbb{Z}_n,$ where $r \geq 2$ and $\text{gcd}(p, n)=1.$ Then also,  $\gamma(\mathcal{G}_{E}^{**} (G ))= \text{number of distinct p-order subgroups of } G_1 .$ 
\end{theorem}

We now shift our attention to the case when $\mathcal{G}_E^{**}(G)$ is connected and here we have the following result. 

\begin{theorem}\label{domination number of enhced pwr raph for connected_nilpotent}	
 Let $G_1$ be a product of non-cyclic  $p$-groups, that is, of the following form: \[G_1 = P_1 \times P_2 \times \dots \times  P_m\] where $m \geq 2$ and for each $ i \in [m]$, $P_i$ is a  $p_i$-group which is neither cyclic nor generalized quarternion. Let $s_i$ be the number of distinct $p_i$-order subgroups of $G_1.$ Then,
	%				\times \mathbb{Z}_{p^{t_{21}}_2}\times\mathbb{Z}_{p^{t_{22}}_2}\times\cdots\times\mathbb{Z}_{p^{t_{2k_2}}_2}\times\cdots 
	\begin{equation}
	\label{eqn:dominating_connected} \gamma(\mathcal{G}_E^{**}(G_1)) = 
	\min _{ 1 \leq i \leq m}    s_i.
	\end{equation}
	Let $G = G_1 \times \mathbb{Z}_n$ with $\text{gcd}(n,|G_1|)=1$, then also 
	\begin{equation}
	\label{eqn:dominating_connected_2}
	\gamma(\mathcal{G}_E^{**}(G)) = 
	\displaystyle  \min _{ 1 \leq i \leq m}    s_i.
	\end{equation}
\end{theorem} 	

Last but not the least, as an application of  
finding the dominating vertices of the enhanced power graph we mention a spectral theoretic connection and prove Theorem 
\ref{Thm: mul spec rad for nilpotent grp}.  
\section{Preliminaries}\label{Sec:Preliminaries} 
In this section, we first recall some earlier known  results on enhanced power graphs which we need throughout the paper.
In \cite{enhancedpwrgrapbb3} and \cite{bera-dey-sajal}, Bera et al. studied many interesting properties of enhanced power graphs of finite groups. In fact, they proved the following:
\begin{lemma}[Theorem 2.4, \cite{enhancedpwrgrapbb3}]\label{enhd coplte iff cyclic}
The enhanced power graph $\mathcal{G}_E(G)$ of the group $G$ is complete if and only if $G$ is cyclic.
\end{lemma}
\begin{lemma}[Theorem 1.1, \cite{bera-dey-sajal}]\label{vc=1 for abln and non abln group}
	Let $G$ be a finite $p$-group such that $G$ is neither cyclic nor generalized quaternion group. Then $\kappa(\mathcal{G}_E(G))=1.$	
\end{lemma}
\begin{lemma}[Lemma 2.5,  \cite{bera-dey-sajal}]
	\label{lemma:any_odtwo_commuting_elements_of_prime_order_adjacent}
	Let $G$ be a finite group and $x, y \in G \setminus
	\{e\}$ be such that $\text{gcd}(\text{o}(x), \text{o}(y)) = 1 $ and $ xy = yx.$ Then, $x \sim y$ in $\mathcal{G}_E^*(G)$. 
\end{lemma}

\begin{lemma}[Lemma 2.6,  \cite{bera-dey-sajal}]
	\label{lema: p and p^i orderd path connd abelong< b>}
	Let $G$ be a $p$-group. Let $a, b$ be two elements of $G$ of order $p, p^i (i\geq 1)$ respectively. If there is a path between $a$ and $b$ in $\mathcal{G}_E^*(G),$ then $\langle a\rangle \subseteq \langle b\rangle.$	In particular, if both a and b have order p, then, $\langle a \rangle = \langle b \rangle.$
\end{lemma}

\begin{lemma}[Theorem 3.1, \cite{bera-dey-sajal}]
	\label{thm:components_of_G_abelian_earlier}
	Let G be a finite abelian $p$-group. Suppose that
	$$G = Z_{p^{t_1}}
	\times  Z_{p^{t_2}}
	\times  \cdots \times 
	Z_{p^{t_r}}.$$ where $r \geq 2$ and $1 \leq t_1 \leq t_2 \leq \cdots \leq t_r.$
	Then, the number of components of $\mathcal{G}_{E}^{**} (G )$  
	is $\frac{p^r-1}{p-1}.$ 
\end{lemma}
%\begin{lemma}[Theorem 1.4, \cite{bera-dey-sajal}]
%	\label{thm: all dom of G times Zn}
%	Let $G_1$ be a non-cyclic abelian group such that $G_1$ has no cyclic sylow subgroup. If $n\in \mathbb{N}, \text{ and gcd}(|G_1|, n)=1$, then $\text{Dom}(\mathcal{G}_E({G_1\times \mathbb{Z}_n}))=\{(e_{G_1}, x), \text{ where } x \text{ is any element of } \mathbb{Z}_n   \}.$ 
%\end{lemma}
\begin{lemma}[Theorem 3.3, \cite{enhancedpwrgrapbb3}]\label{dom of gen Q_2^n in enhced}
Let $G$ be a non-abelian $2$-group. Then the enhanced power graph $\mathcal{G}_e(G)$ is dominatable  if and only if $G$ is generalized quarternion group. In this case identity and the unique element of order $2$ are dominating vertices.
\end{lemma}
We next prove some important results which are used to prove our main theorems.
\begin{theorem}\label{thm: dom, G any grp product cyclic grp}
Let $G$ be a finite group and $n\in \mathbb{N}.$ If $gcd(|G|, n)=1$, then $\{(e, a): a\in \mathbb{Z}_n\}\subseteq\text{Dom}(\mathcal{G}_E({G\times \mathbb{Z}_n})).$
\end{theorem}
\begin{proof}
We show that $(e, a)$ is a dominating vertex, where $e$ is the identity element of the group $G$ and $a$ is an arbitrary element of the group $\mathbb{Z}_n.$ Consider an arbitrary vertex $(g, b)$  of the graph  $\mathcal{G}_E({G\times \mathbb{Z}_n}).$
	
Case 1: Let $g=e.$ Suppose $a'$ is a generator of the cyclic group $\mathbb{Z}_n.$ Now, $a, b\in \mathbb{Z}_n$ and $a'$ is a generator of $\mathbb{Z}_n$ implies that $(e, b), (e, a)\in \langle(e, a')\rangle.$ As a result, $(e, b)\sim (e, a)$ in $\mathcal{G}_E({G\times \mathbb{Z}_n}).$
	
Case 2: Let $g\neq e$ and $b=e',$ where $e'$ is the identity of the group $\mathbb{Z}_n.$ We show that $(g, e'), (e, a)\in \langle(g, a)\rangle.$ Let $\text{o}(g)=m.$ Then $(e, a^m)=(g, a)^m\in \langle(g, a)\rangle.$ Now $\text{gcd}(|G|, n)=1$ implies that $\text{gcd}(m, n)=1.$ So $(e, a)^m=(e, a^m)$ is a generator of the cyclic group $\langle(e, a)\rangle.$  Therefore, $(e, a)\in \langle(e, a^m)\rangle.$ Again $(e, a^m)\in \langle(g, a)\rangle$ implies that $(e, a)\in \langle(g, a)\rangle.$ Here we show that $(g, e')\in \langle(g, a)\rangle.$ Now $\text{gcd}(m, n)=1$ implies that $n^{\phi(m)}=mk+1, k\in \mathbb{N} (\text{ by the Euler's theorem }).$ Therefore, $(g, a)^{n^{\phi(m)}}=(g^{n^{\phi(m)}}, e')=(g^{mk+1}, e')=(g, e').$ Hence, $(g, e')\in \langle (g, a) \rangle.$
	
Case 3: Let $g\neq e $ and $b\neq e'.$ We show that $(g, b), (e, a)\in \langle(g, a')\rangle.$ Already we have proved that $(g, e')\in \langle(g, a)\rangle$ and $(e, a)\in \langle(g, a)\rangle.$ Now $a'$ is a generator of $\mathbb{Z}_n$ implies $(e, b)\in \langle(e, a')\rangle \subseteq \langle(g, a')\rangle$ (in fact, $\text{gcd}(m, n)=1\Rightarrow k_1m+k_2n=1.$ Therefore, $(g, a')^{k_1m}=(e, a')).$ Hence $ (g, b)=(g, e')(e, b)\in \langle(g, a')\rangle.$ This completes the proof.
\end{proof}
\begin{proposition}\label{prop: dom for product of groups}
Let $G_1, \cdots, G_r$ be finite groups. Then $\text{Dom}(\mathcal{G}_E(G_1\times\cdots\times G_r))\subseteq\text{Dom}(\mathcal{G}_E(G_1))\times \cdots\times \text{Dom}(\mathcal{G}_E(G_r)).$	
\end{proposition}
\begin{proof}
Let $(a_1, \cdots, a_r)$ be a dominating vertex of the graph $\mathcal{G}_E(G_1\times\cdots\times G_r).$ Let $(b_1, \cdots, b_r)$ be an arbitrary vertex of  $\mathcal{G}_E(G_1\times\cdots\times G_r).$ Now $(a_1, \cdots, a_r)$ is a dominating vertex, so there exists $(c_1, \cdots, c_r)\in G_1\times\cdots\times G_r$ such that $(a_1, \cdots, a_r), (b_1, \cdots, b_r)\in \langle (c_1, \cdots, c_r)\rangle.$ Therefore, for each $i\in[r],$ we have $a_i, b_i\in \langle c_i\rangle.$ As a result, $a_i\sim b_i.$  As $b_i$ is arbitrary, so $a_i$ is a dominating vertex of $\mathcal{G}_E(G_i).$  Hence, $(a_1, \cdots, a_r)\in \text{Dom}(\mathcal{G}_E(G_1))\times \cdots\times \text{Dom}(\mathcal{G}_E(G_r)).$ 	
\end{proof}
\begin{rem}
The equality of Proposition \ref{prop: dom for product of groups} is not strict, in general. One can take $G_1 = G_2 = \mathbb{Z}_p$ and in that case, $Dom(\mathcal{G}_E{(G_1 \times G_2)}) =\{ (e,e)\}$, where as $Dom(\mathcal{G}_E(G_1)) \times 
Dom(\mathcal{G}_E(G_2)) = \mathbb{Z}_p \times \mathbb{Z}_p.$	
\end{rem}
Now we recall the following result which is already used  to prove Theorem 2.3 of \cite{bera-Linegraph-powergraph}.  
\begin{lemma}\label{lemma: no of distinct cyclic subgroups orer p geq 3}
Let $G$ be a $p$-group. Then the number of distinct $p$-ordered cyclic subgroup is either $1$ or greater than equal to $3.$ 	
\end{lemma}
\begin{proof}	
Let $H_1=\langle v_1^{(p)}\rangle, \cdots, H_{\ell}=\langle v_{\ell}^{(p)}\rangle$ be the collection of all distinct cyclic subgroups of order $p$ of $G.$ We prove that either $\ell=1$ or $\ell\geq 3.$ Let $\ell=2.$ Since $G$ is $p$-group, the center of $G$ is $Z(G)$ is non trivial. Therefore, $Z(G)$ has a subgroup $H_{i_1}=\langle v_{i_1}^{(p)}\rangle$ (say) of order $p.$ Let $H_{i_2}=\langle v_{i_2}^{(p)}\rangle$ be another cyclic subgroup of order $p,$ (as $\ell=2, \text{ we can choose more than one subgroup of order } p).$ Then $v_{i_1}^{(p)}v_{i_2}^{(p)}=v_{i_2}^{(p)}v_{i_1}^{(p)}$ and $\text{o}(v_{i_1}^{(p)}v_{i_2}^{(p)})=p.$ Take $H_{i_3}=\langle v_{i_1}^{(p)}v_{i_2}^{(p)}\rangle$ and it is easy to see that $H_1\neq H_3$ and $H_2\neq H_3.$ This completes the proof.
\end{proof}
\begin{lemma}\label{lemma: v1divides v2 and adjacent, so v1 belongs to v2}
Let $G$ be any group. Let $v_1, v_2\in V(\mathcal{G}_E(G))$ such that $\text{o}(v_1)=m_1,\text{o}(v_2)=m_2$ and $m_1$ divides $m_2.$ Then $v_1\sim v_2$ in $\mathcal{G}_E(G)$ if and only if $v_1\in \langle v_2\rangle.$  	
\end{lemma}
\begin{proof}
Let $v_1\in \langle v_2\rangle.$ Then clearly $v_1\sim v_2$ in $\mathcal{G}_E(G).$ Conversely suppose that $v_1\sim v_2$ in $\mathcal{G}_E(G).$ So, there exists $v\in G$ such that $v_1, v_2\in \langle v\rangle.$ Clearly, for each $i\in \{1, 2\},$ $\langle v_i\rangle$ is the unique cyclic subgroup of order $m_i$ of $\langle v\rangle.$ Again, $m_1$ divides $m_2,$ so $\langle v_2\rangle$ must have a subgroup say $H$ of order $m_1.$ Now $\langle v_2\rangle$ is a subgroup of $\langle v\rangle.$ Therefore, $H$ is a subgroup of $\langle v\rangle$ of order $m_1.$ Again $\langle v_1\rangle$ is the unique subgroup of $\langle v\rangle$ of order $m_1.$ So, $H=\langle v_1\rangle.$ Hence $v_1\in \langle v_2\rangle.$
\end{proof}
%In this paper, our primary focus is on finite abelian groups.
%In  \cite[Theorem 2.4]{enhancedpwrgrapbb3}  Bera et al. proved that for a finite abelian group $G$,
%the enhanced power graph $\mathcal{G}_e^{**}(G)$ of the group $G$ is complete if and only if $G$ is cyclic. Therefore, we always assume $G$ to be non-cyclic. 
%Clearly, the minimum possiblity for domination number of the proper enhanced power graph for a group $G$ is $ \geq 2$.
\section{Dominating vertices of $\mathcal{G}_E(G)$ when $G$ is nilpotent}
\label{sec:dom_vertices}

In Section 5, we will focus on the connectivity of the proper enhanced power graph of finite nilpotent group. To do so first we have to characterize all the dominating vertices of the enhanced power graph. In this portion, we completely classify the dominating vertices of the graph $\mathcal{G}_E(G)$ when $G$ is nilpotent.
  
We denote by $Q_{2^k},$ a generalized quaternion group of order $2^k.$ Let $G_1$ be a nilpotent group having no sylow subgroups which are either cyclic or generalized quaternion. Suppose $e, e''$ are the identities of $G_1$ and $Q_{2^k}$ respectively.  Let $D_1=\{(e, x, e''): x\in \mathbb{Z}_n\}$ and $D_2=\{(e, x, y): x \in \mathbb{Z}_n, y\in Q_{2^k} \text{ and } \text{o}(y)=2\}.$ Then  the following theorem completely characterizes the dominating vertices of the enhanced power graph of any nilpotent group.
\begin{theorem}\label{Thm: Dom set for nilpotent grp}
Let $G$ be a finite nilpotent group. Then 
\begin{equation*}
\text{Dom}(\mathcal{G}_E(G))= 
	\begin{cases}
	\{e\}, & \text{ if } G=G_1 \\
	\{(e, x): x \in \mathbb{Z}_n\},& \text{ if } G=G_1\times \mathbb{Z}_n\text{ and } \text{gcd}(|G_1|, n)=1 \\ 
	\{(e, e''), (e, y)\}, & \text{ if } G=G_1 \times Q_{2^k}\text{ and } \text{gcd}(|G_1|, 2)=1 \\
	D_1\cup D_2, & \text{ if } G=G_1\times \mathbb{Z}_n\times Q_{2^k} \text{ and } \\&\text{gcd}(|G_1|, n)=\text{gcd}(|G_1|, 2)=\text{gcd}(n, 2)=1.
	\end{cases}
	\end{equation*}
\end{theorem}
Any abelian group is nilpotent. So, by Theorem \ref{Thm: Dom set for nilpotent grp}, we immediately get the following
corollary.
\begin{corollary}[Theorem 1.4, \cite{bera-dey-sajal}]\label{thm: all dom of G times Zn}
Let $G$ be a non-cyclic abelian group such that $G$ has no cyclic sylow subgroup. If $n\in \mathbb{N}, \text{ and gcd}(|G|, n)=1,$ then $\text{Dom}(\mathcal{G}_E({G\times \mathbb{Z}_n}))=\{(e, x), \text{ where } e \text{ is the identity of } G  \text{ and } x\in\mathbb{Z}_n \}.$ 
\end{corollary}
According to the structure (described as in Subsection \ref{Defn: Nilpotent group}) of nilpotent group, we split the proof of Theorem \ref{Thm: Dom set for nilpotent grp} into following $4$ propositions. 
\begin{proposition}\label{prop: dom; Each sylow is neither cyclic nor generalized quaternion group}
Let $G_1$ be a nilpotent group such that $G_1$ has no sylow subgroups which are either cyclic or generalized quaternion. Then $\text{Dom}(\mathcal{G}_E(G_1))=\{e\},$ the identity of $G_1.$ 
\end{proposition}
\begin{proof}
$G_1$ is nilpotent so, $G_1\cong P_1\times \cdots\times P_r,$ where each $P_i$ is a sylow subgroup of order $p_i^{t_i}$ of $G_1.$ Let $g(\neq e)$ be a dominating vertex of the graph $\mathcal{G}_E(G_1).$ Suppose $p_i$ is a prime divisor of $\text{o}(g).$ Now it is given that each $P_i$ is neither cyclic nor generalized quaternion. So, the number of distinct cyclic subgroup of order $p_i$ in $P_i$ is greater than $1.$ Let $H_1=\langle h\rangle$ and $H_2=\langle h'\rangle$ be two distinct cyclic subgroups of order $p_i$ of $P_i.$ Since $g$ is a dominating vertex, then $g\sim h$ and $g\sim h'$ in $\mathcal{G}_E(G_1).$  Therefore, there exists two cyclic subgroups $K_1, K_2$ in $G_1$ such that $g, h\in K_1$ and $g, h'\in K_2.$ Now $g, h\in K_1,$ and $K_1$ is cyclic, so $h\in \langle g\rangle.$ Similarly, $h'\in \langle g\rangle.$ Therefore, $\langle g\rangle$ contains two distinct cyclic subgroup of order $p_i.$ A contradiction. Hence $\text{Dom}(\mathcal{G}_E(G_1))=\{e\}.$
\end{proof}
\begin{proposition}\label{Prop: all dom of G times Zn}	
Let $G_1$ be a nilpotent group having no sylow subgroups which are either cyclic or generalized quaternion. If $n\in \mathbb{N}, \text{ and gcd}(|G_1|, n)=1,$ then $\text{Dom}(\mathcal{G}_E({G_1\times \mathbb{Z}_n}))=\{(e, g): e \text{ is the identity of } G_1 \text{ and } g \in \mathbb{Z}_n\}.$ 
\end{proposition}
\begin{proof}
By Theorem \ref{thm: dom, G any grp product cyclic grp},  $\{(e, g): g\in \mathbb{Z}_n\}\subseteq \text{Dom}(\mathcal{G}_E(G_1\times \mathbb{Z}_n)).$ Let $(a, g)$ be a dominating vertex of $\mathcal{G}_E(G_1\times \mathbb{Z}_n),$ where $a\neq e.$ Then, $(a, g)\sim (x, y)$ for all $x\in G_1$ and $y\in \mathbb{Z}_n.$ So, there exists a cyclic subgroup $\langle (c_x, d_y)\rangle$ such that $(a, g), (x, y)\in \langle (c_x, d_y)\rangle.$ As a  result, $a, x\in \langle c_x\rangle$ and $g, y\in \langle d_y\rangle.$ Consequently, $a\sim x$ and $g\sim y.$ Since $x$ is arbitrary, we can say that $a (\neq e)$ is a dominating vertex of $\mathcal{G}_E(G_1).$ This contradicts Theorem \ref{prop: dom; Each sylow is neither cyclic nor generalized quaternion group}. Hence the result.
\end{proof}
\begin{proposition}\label{prop: dom; nilpotent grp times quaternion grp}
Let $G_1$ be a nilpotent group having no sylow subgroups which are either cyclic or generalized quaternion. If $n\in \mathbb{N}, \text{ and gcd}(|G_1|, 2)=1,$ then $\text{Dom}(\mathcal{G}_E({G_1\times Q_{2^k}}))=\{(e, e''), (e, y) \}, \text{ where } e, e'' \text{ are the identities of } G_1, Q_{2^k} \text{ respectively  and }  \text{o}(y)=2.$ 	
\end{proposition}
\begin{proof}
We know that $Q_{2^k}$ has a unique subgroup  $\langle y \rangle$ (say) of order $2.$ Now $\text{gcd}(2, |G_1|)=1,$ so $G_1\times Q_{2^k}$ has a unique subgroup of order $2,$ namely $\langle(e, y)\rangle.$ We show that $(e, y)$ is a dominating vertex of the graph $\mathcal{G}_E(G_1\times Q_{2^k}).$ Let $(a, b)$ be an arbitrary vertex of $\mathcal{G}_E(G_1\times Q_{2^k}).$ 
	
Case 1: Let $b\neq e'',$ the identity of the group $Q_{2^k}.$ Then $\text{o}(a, b)$ has a factor $2^{\ell}, \ell\geq 1.$ So, $\langle (a, b)\rangle$ has a cyclic subgroup of order $2.$ Now $\langle(e, y)\rangle$ is the unique subgroup of order $2$ implies that $(e, y)\in \langle (a, b)\rangle.$ Hence $(e, y)\sim (a, b).$  
	
Case 2: Let $b=e''.$  Here we show that $(a, e'')\in \langle(a, b')\rangle,$ for some $b'\in Q_{2^k}$ such that $\text{o}(b')=2^{r}, r\geq 1.$ It is given that $\text{gcd}(2^k, |G_1|)=1,$ so there exists $k_1, k_2$ such that $2^kk_1+\text{o}(a)k_2=1.$ Now it is cleared that $(a, b')^{2^kk_1}=(a, e'').$ So, $(a, e'')\in \langle (a, b')\rangle.$ As before we can show that $(e, y)\in \langle (a, b')\rangle.$ As a result, $(e, y)\sim (a, e'').$ Hence, $(e, y)$ is a dominating vertex. Now by Lemma \ref{dom of gen Q_2^n in enhced} and Proposition \ref{prop: dom; Each sylow is neither cyclic nor generalized quaternion group} we can say that $\text{Dom}(\mathcal{G}_E({G_1\times Q_{2^k}}))=\{(e, e''), (e, y) \}, \text{ where } e, e'' \text{ are the identities of } G_1, Q_{2^k} \text{ respectively  and }  \text{o}(y)=2.$  	
\end{proof}
\begin{proposition}\label{prop:dom_cyclic_quarternion_normal}
Let $G_1$ be a nilpotent group having no sylow subgroups which are either  cyclic or generalized quaternion. If $n\in \mathbb{N}, \text{ and gcd}(|G_1|, n)=\text{gcd}(|G_1|, 2)=\text{gcd}(n, 2)=1,$ then $\text{Dom}(\mathcal{G}_E({G_1\times\mathbb{Z}_n\times Q_{2^k}}))=\{(e, x, e''): x\in \mathbb{Z}_n\}\cup\{(e, x, y): x \in \mathbb{Z}_n \text{ and } y\in Q_{2^k} \text{ with } \text{o}(y)=2\}.$ 	
\end{proposition}
\begin{proof}
We show that each element of the set $D=D_1\cup D_2$ is a dominating vertex, where $D_1=\{(e, x, e''): x\in \mathbb{Z}_n\} \text{ and } D_2=\{(e, x, y): x \in \mathbb{Z}_n \text{ and } y\in Q_{2^k} \text{ with } \text{o}(y)=2\}.$ Let $(a, b, c)$ be an arbitrary vertex of $\mathcal{G}_E({G_1\times\mathbb{Z}_n\times Q_{2^k}}).$ First we show that $(e, x, e'')$ and $(a, b, c)\in \langle (a, x', c)\rangle.$ Here we consider two cases.
	
Case 1: Let $c\neq e''.$ It is cleared that $(e, x, y)=(e, x, e'')(e, e', y),$ where $e'$ is the identity of $\mathbb{Z}_n.$ Let $|G_1|=m$ and $\text{o}(c)=2^{\ell}, \ell\geq 1.$ Now $\text{gcd}(2^{\ell}, n)=1$ and $\text{gcd}(m, n)=1.$ So, there exist $k_1, k_2$ and $r_1, r_2$ such that $k_12^{\ell}+k_2n=1$ and $r_1m+r_2n=1.$ Now 
	\begin{align*}
	(a, x', c)^{(2^{\ell}k_1)(r_1m)}=&(a^{2^{\ell}k_1}, x'^{2^{\ell}k_1}, c^{2^{\ell}k_1})^{r_1m}\\
	=&(a^{2^{\ell}k_1}, x', e'')^{r_1m}\\
	=&(e, x', e'')
	\end{align*}
	So, $(e, x', e'')\in \langle (a, x', c)\rangle.$ Again $x'$ is a generator of $\mathbb{Z}_n,$ therefore, $(e, x, e'')\in \langle (e, x', e'')\rangle\subseteq \langle (a, x', c)\rangle.$ Also, it is given that $\text{gcd}(2^{\ell}, m)=1.$ So there exist $s_1, s_2$ such that $2^{\ell}s_1+ms_2=1.$ Therefore,  
	\[(a, x', c)^{(k_2n)(2^{\ell}s_1)}=(a, e', e'') \text{ and } (a, x', c)^{(k_2n)(s_2m)}=(e, e', c).\] Also, $(e, b, e'')\in \langle (e, x', e'')\rangle $ (as $b\in \mathbb{Z}_n$ and $x'$ is a generator of $\mathbb{Z}_n).$ Moreover, it is proved that $\langle(e, x', e'')\rangle\subseteq \langle (a, x', c)\rangle.$ Therefore, $(e, b, e'')\in \langle (a, x', c)\rangle.$  Now, $(a, b, c)=(a, e', e'')(e, b, e'')(e, e', c).$ As a result, if $c\neq e'',$ then $(a, b, c)\in \langle (a, x', c)\rangle.$ Hence in this case, $(e, x, e'')\sim (a, b, c).$ 
	
	Now we show that $(e, x, y)\in \langle (a, x', c)\rangle.$
	Clearly, $(e, x, e'')(e, e', y)=(e, x, y).$ Already we have shown that $(e, x, e'')\in \langle(a, x', c)\rangle.$ So, it is enough to show that $(e, e', y)\in \langle (a, x', c)\rangle.$ Note that $y$ is the unique element of order $2$ in $Q_{2^k}$ and $\text{o}(c)=2^{\ell}.$ Therefore, $c^r=y,$ for some $r.$ Now it is easy to see that $(a, x', c)^{(rs_2m)(k_2n)}=(e, e', y).$ Thus $(e, x, y)=(e, x, e'')(e, e', y)\in \langle (a, x', c)\rangle.$ Consequently $(e, x, y)\sim (a, b, c).$
	
	Case 2: Let $c=e''.$ Then $(a, b, e'')=(a, e', e'')(e, b, e'').$ In this case, continuing the same way as Case 1, we can show that $(e, x, e''), (e, x, y), (a, b, e'')\in \langle (a, x', y)\rangle.$ Thus $(e, x, e'')\sim (a, b, e'')$ and $(e, x, y)\sim (a, b, e'').$ Hence $D\subseteq\text{Dom}(\mathcal{G}_E({G_1\times\mathbb{Z}_n\times Q_{2^k}})).$ Now by Lemma \ref{dom of gen Q_2^n in enhced} and Proposition \ref{prop: dom for product of groups},  $\text{Dom}(\mathcal{G}_E({G_1\times\mathbb{Z}_n\times Q_{2^k}}))=D.$     		
\end{proof}

\begin{proof}[Proof of Theorem \ref{Thm: Dom set for nilpotent grp} and Theorem \ref{Thm: G' nilpotent dominatability iff}] Theorem \ref{Thm: Dom set for nilpotent grp} follows from Propositions \ref{prop: dom; Each sylow is neither cyclic nor generalized quaternion group},  \ref{Prop: all dom of G times Zn}, \ref{prop: dom; nilpotent grp times quaternion grp}, and  \ref{prop:dom_cyclic_quarternion_normal}. Theorem \ref{Thm: G' nilpotent dominatability iff} directly follows from Theorem \ref{Thm: Dom set for nilpotent grp}. 
\end{proof}

%Now any abelian group $G$ is nilpotent. Moreover, if $G$ has cyclic sylow subgroups, then $G=G_1\times \mathbb{Z}_n,$ where $\text{gcd}(|G_1|, n)=1,$ and $G_1$ is abelian having no sylow subgroups which are either cyclic or generalized quaternion. So, as an immediate consequence of Theorem \ref{Thm: G' nilpotent dominatability iff} we have the following:
%
%\begin{corollary}[Theorem 3.2, \cite{enhancedpwrgrapbb3}]\label{thm:bb3 ehced, dom iff cylic syllow}
%	Let $G$ be a finite abelian group. Then $\mathcal{G}_E(G)$
%	is dominatable if and only if $G$ has a cyclic Sylow subgroup. 
%\end{corollary}

\section{Connectivity and diameter of proper enhanced power graph}
\label{sec:connectivity_diameter} 
Finding the vertex connectivity for the power graph and enhanced power graph of a group has been a very interesting problem for the last decade. Many researchers have attempted and found out good bounds for the power graphs. For the power graph of a cyclic group, Chattopadhyay et. al in \cite{chatto-sahoo-patro, chatto-sahoo-patro-ii} found the exact vertex connectivity for most of the cyclic groups and has given an upper bound for the rest. Chattopadhyay et al. in \cite{chatto-sahoo-patro-iii} have given exact values for the vertex connectivity of particular kinds of some nilpotent groups. Bera et al. in \cite[Theorem 1.6]{bera-dey-sajal} proved the following upper bound on the vertex connectivity for the enhanced power graph of an abelian group. 
\begin{theorem}\label{vc of enhced pwr raph for grnrral abelian grp}	
	Let $G$ be a non-cyclic abelian group such that \[G\cong\mathbb{Z}_{p^{t_{11}}_1}
	\times\cdots\times\mathbb{Z}_{p^{t_{1k_1}}_1}\times \mathbb{Z}_{p^{t_{21}}_2}\times\cdots\times\mathbb{Z}_{p^{t_{2k_2}}_2}\times\cdots\times
	\mathbb{Z}_{p^{t_{r1}}_r}\times\cdots\times\mathbb{Z}_{p^{t_{rk_r}}_r},\] where $k_i\geq 1 $ and $1\leq t_{i1}\leq t_{i2}\leq\cdots\leq t_{ik_i} $, for all $i\in [r]. $ Then \[\kappa(\mathcal{G}_E(G))\leq p_1^{t_{11}}p_2^{t_{21}}\cdots p_r^{t_{r1}}-\phi(p_1^{t_{11}}p_2^{t_{21}}\cdots p_r^{t_{r1}}).\]
\end{theorem}
In this section, we prove an improved bound on the vertex connectivity for the enhanced power graph of a finite abelian group. Also, we derive the exact value of the vertex connectivity of the enhanced power graphs of some particular kind of nilpotent groups. Besides, we classify all nilpotent groups $G$ for which $\mathcal{G}_E^{**}(G)$ is connected and find out their diameters.  
\begin{theorem}
\label{thm:at_least_r_components}
	Let $G$ be an abelian group such that  $G \cong G_1  \times \mathbb{Z}_n$ and $\text{gcd}(|G_1|, n)=1.$ Let, $S \subset V(\mathcal{G}_E(G))$  such that
	$\mathcal{G}_{E}(G_1 \setminus S)$ has exactly $r$ components $C_1, C_2, \cdots, C_r.$ Then, 	$\mathcal{G}_{E}(G \setminus (S \times \mathbb{Z}_n))$ has at least $r$ components. In particular, if $\mathcal{G}_{E}(G_1 \setminus S)$ is disconnected, $\mathcal{G}_{E}(G \setminus (S \times \mathbb{Z}_n))$ is also disconnected. 
\end{theorem}
\begin{proof}
	Define 
	$f:C(\mathcal{G}_{E}(G_1 \setminus S)) \mapsto C(\mathcal{G}_{E}(G \setminus (S \times \mathbb{Z}_n)))$ by
	$$f(C_i)= C_i \times  \mathbb{Z}_n .$$
	It is enough to show that there is no path in between $C_i \times  \mathbb{Z}_n$ and 
	$C_j \times  \mathbb{Z}_n$ for $1 \leq i < j \leq r.$ Let, there exists 
	an path between $(a_1, b_1)$ and $(a_2, b_2)$ where $a_1 \in C_i$, $a_2 \in C_j$ and $b_1, b_2 \in \mathbb{Z}_n.$ 
	If possible, let $(a_1, b_1) \sim (c_1, d_1) \sim (c_2, d_2) \sim \cdots \sim (c_{m-1}, d_{m-1}) \sim (a_2, b_2)$ in $\mathcal{G}_E(G \setminus (S \times \mathbb{Z}_n) )$ where $c_1, c_2,\dots, c_{m-1} \in G_1$ and $d_1, d_2, \dots,d_{m-1} \in \mathbb{Z}_n$. Then $c_1, c_2, \cdots, c_{m-1}$ must be non-zero elements of $G_1$. This proves that $a_1$ and $a_2$ are connected by a path in $\mathcal{G}_E(G_1 \setminus S)$ which contradicts the fact that $C_i$ and $C_j$ are distinct connected components of $\mathcal{G}_{E} (G_1 \setminus S).$
	%If possible, there exists 
	%  	a path between $(a_1, b_1)$ and $(a_2, b_2)$ where $a_1 \in C_i$, $a_2 \in C_j$ and $b_1, b_2 \in \mathbb{Z}_n.$ 
	%  	If possible, let $(a_1, b_1) \sim (c_1, d_1) \sim (c_2, d_2) \sim \cdots \sim (c_{m-1}, d_{m-1}) \sim (a_2, b_2)$ in $\mathcal{G}_{e}(G \setminus (S \times \mathbb{Z}_n))$, where $c_1, c_2,\dots, c_{m-1} \in G_1 \setminus S$ and $d_1, d_2, \dots,d_{m-1} \in \mathbb{Z}_n$. This proves that $a_1$ and $a_2$ are connected by a path in $\mathcal{G}_{E}(G_1 \setminus S)$ which contradicts the fact that $C_i$ and $C_j$ are distinct connected components. 
	The proof is complete.
\end{proof}
We are now in a position to prove the main result of this section. \begin{theorem}\label{improved vc of enhced pwr raph for grnrral abelian grp}	
	Without loss of generality, we can assume that $G$ is non-cyclic abelian group such that
	\begin{eqnarray}\label{eqn:G_form} 
	G  & \cong & \mathbb{Z}_{p^{t_{11}}_1}\times\cdots\times\mathbb{Z}_{p^{t_{1k_1}}_1}\times \mathbb{Z}_{p^{t_{21}}_2}\times\cdots\times\mathbb{Z}_{p^{t_{2k_2}}_2}\times\cdots\times
	\mathbb{Z}_{p^{t_{r1}}_r}\times\cdots\times\mathbb{Z}_{p^{t_{rk_r}}_r} \times  \nonumber \\
	& &  \mathbb{Z}_{p^{t_{r+1 1}}_{r+1}} \times 
	\mathbb{Z}_{p^{t_{r+2 1}}_{r+2}} \times	 \cdots \times  
	\mathbb{Z}_{p^{t_{s 1}}_{s}}. \nonumber
	\end{eqnarray}	
	where $k_i\geq 2 $ and $1\leq t_{i1}\leq t_{i2}\leq\cdots\leq t_{ik_i} $, for all $i\in [r]. $
	We then have, 
	$$\kappa(\mathcal{G}_E(G))\leq p^{t_{r+1 1}}_{r+1}
	p^{t_{r+2 1}}_{r+2} \cdots p^{t_{s 1}}_{s} \left(p_1^{t_{11}}p_2^{t_{21}}\cdots p_r^{t_{r1}}-\phi(p_1^{t_{11}}p_2^{t_{21}}\cdots p_r^{t_{r1}})\right).$$
\end{theorem} 
\begin{proof}
	Let, $n= p^{t_{r+1 1}}_{r+1}
	p^{t_{r+2 1}}_{r+2} \cdots p^{t_{s 1}}_{s}$  and \[G_1= \mathbb{Z}_{p^{t_{11}}_1}\times\cdots\times\mathbb{Z}_{p^{t_{1k_1}}_1}\times \mathbb{Z}_{p^{t_{21}}_2}\times\cdots\times\mathbb{Z}_{p^{t_{2k_2}}_2}\times\cdots\times
	\mathbb{Z}_{p^{t_{r1}}_r}\times\cdots\times\mathbb{Z}_{p^{t_{rk}}_r}.\]
	Clearly we have $(|G_1|,n)=1.$ By Theorem \ref{vc of enhced pwr raph for grnrral abelian grp}, there exists a set $S$ of cardinality  $p_1^{t_{11}}p_2^{t_{21}}\cdots p_r^{t_{r1}}-\phi(p_1^{t_{11}}p_2^{t_{21}}\cdots p_r^{t_{r1}})$ such that $G_1 \setminus S$ is disconnected. Hence, by Theorem \ref{thm:at_least_r_components}, $\mathcal{G}_{E}(G \setminus (S \times \mathbb{Z}_n))$ is disconnected. Moreover, the cardinality of $|S\times \mathbb{Z}_n|$ is $$|S\times \mathbb{Z}_n|=  p^{t_{r+1 1}}_{r+1}p^{t_{r+2 1}}_{r+2} \cdots p^{t_{s 1}}_{s} |S|$$
	The proof is complete. 
\end{proof}
Let $G$ be a non-cyclic abelian group such that
\begin{eqnarray}\label{eqn:G_form_2} 
G  & \cong & \mathbb{Z}_{p^{t_{11}}_1}\times\cdots\times\mathbb{Z}_{p^{t_{1k_1}}_1}\times \mathbb{Z}_{p^{t_{21}}_2}\times\cdots\times\mathbb{Z}_{p^{t_{2k_2}}_2}\times\cdots\times
\mathbb{Z}_{p^{t_{r1}}_r}\times\cdots\times\mathbb{Z}_{p^{t_{rk_r}}_r} \times  \nonumber \\
& &  \mathbb{Z}_{p^{t_{r+1 1}}_{r+1}} \times 
\mathbb{Z}_{p^{t_{r+2 1}}_{r+2}} \times	 \cdots \times  
\mathbb{Z}_{p^{t_{s 1}}_{s}}. \nonumber
\end{eqnarray}	
and $\alpha(G)$ and $\beta(G)$ be the bounds for $\kappa(\mathcal{G}_E(G))$ in Theorem \ref{vc of enhced pwr raph for grnrral abelian grp} and Theorem \ref{improved vc of enhced pwr raph for grnrral abelian grp} respectively. That is, 
$$\alpha(G)=p_1^{t_{11}}p_2^{t_{21}}\cdots p_r^{t_{r1}}p^{t_{r+1 1}}_{r+1}
p^{t_{r+2 1}}_{r+2} \cdots p^{t_{s 1}}_{s}-\phi(p_1^{t_{11}}p_2^{t_{21}}\cdots p_r^{t_{r1}}p^{t_{r+1 1}}_{r+1}p^{t_{r+2 1}}_{r+2} \cdots p^{t_{s 1}}_{s}) $$ and 
$$\beta(G)= p^{t_{r+1 1}}_{r+1}
p^{t_{r+2 1}}_{r+2} \cdots p^{t_{s 1}}_{s} \left(p_1^{t_{11}}p_2^{t_{21}}\cdots p_r^{t_{r1}}-\phi(p_1^{t_{11}}p_2^{t_{21}}\cdots p_r^{t_{r1}})\right). $$

We at first prove that for any abelian group $G$, the bound $\beta(G)$ is better.
\begin{lemma}
	\label{lem:comparison_of_bounds}
	For all non-cyclic abelian groups $G$, we have $\beta(G) \leq \alpha(G)$. Moreover, equality happens if and only if $G$ do not have any cyclic sylow $p$-subgroup. 
\end{lemma}
\begin{proof}
	For any two positive integers $a$ and $b$, we have $a\phi(b) \geq \phi (ab)$ and equality happens if and only if $a=1$. Now set $a= p^{t_{r+1 1}}_{r+1}
	p^{t_{r+2 1}}_{r+2} \cdots p^{t_{s 1}}_{s}$ and $b=  p_1^{t_{11}}p_2^{t_{21}}\cdots p_r^{t_{r1}}$ and the proof is complete.  
\end{proof}
Looking at Lemma \ref{lem:comparison_of_bounds}, one may think that the improved bound $\beta(G)$ is not a much better bound than the previously known bound $\alpha(G)$ but in the next remark we show that the differences can be much larger.  
\begin{rem}
	\label{rem:vc_enhcd_why_good}
	Theorem \ref{improved vc of enhced pwr raph for grnrral abelian grp} clearly gives a much a better bound than Theorem \ref{vc of enhced pwr raph for grnrral abelian grp} when $G$ has atleast one sylow $p$-subgroup. As an instance, one can take $$G= \mathbb{Z}_3 \times \mathbb{Z}_9 \times  \mathbb{Z}_5 \times \mathbb{Z}_{25}  \times \mathbb{Z}_7 \times  \mathbb{Z}_{49} \times  \mathbb{Z}_{13}.$$ In this case, we have 
	$$\alpha(G)=3.5.7.13- \phi (3.5.7.13)= 789 \hspace{3 mm} \mbox{and} \hspace{3 mm} \beta(G)= 3.5.7 (13 - \phi(13))= 105. $$
	Indeed, if we take 
	$$G= \mathbb{Z}_3 \times \mathbb{Z}_9 \times  \mathbb{Z}_5 \times \mathbb{Z}_{25}  \times \mathbb{Z}_7 \times  \mathbb{Z}_{49} \times  \mathbb{Z}_{p}$$ where $p$ is a large prime, then $\alpha(G) = 3.5.7.p - \phi (3.5.7.p) $ which is very very large, where as $\beta(G)$ remains the same. 
	%This shows that the ratio $\alpha(G) / \beta(G)$ can be as high as we want. 
\end{rem}
In this place, we want to give the exact value of vertex connectivity  of enhanced power graphs of some particular nilpotent groups.    
\begin{theorem}\label{thm: nilpotent no sylow subgroups cyclic quaternion vc=1 iff}
Let $G_1$ be a nilpotent group having no sylow subgroups which are either cyclic or generalized quaternion. Then $\kappa(\mathcal{G}_E(G_1))=1$ if and only if $G_1$ is $p$-group.	
\end{theorem}
\begin{proof}
	Let $G_1$ be $p$-group such that $G_1$ is neither cyclic nor generalized quaternion. Then, by Theorem \ref{vc=1 for abln and non abln group}, $\kappa(\mathcal{G}_E(G_1))=1.$
	
	For the converse part, let $G_1$ be a finite group which is not a $p$-group. Let, $p_1, p_2, \cdots, p_r (r\geq 2)$ be the distinct prime factors of $|G_1|.$ Let, $a, b \in G_1$ and $\text{o}(a)= p_1^{k_1} p_2^{k_2} \cdots p_r^{k_r}$
	and $\text{o}(b)=p_1^{s_1} p_2^{s_2} \cdots p_r^{s_r}.$ We consider the following cases: \\
	Case 1: There exists distinct $i$ and $j$  with 
	$k_i \neq 0$ and $s_j \neq 0.$ So, we have $a', b'\in G_1$ such that $\text{o}(a')=p_i, \text{o}(b')=p_j$ and $a'\in \langle a\rangle$ and $b'\in \langle b\rangle.$ As $G_1$ nilpotent, we have $a'b'=b'a'$ (as $p_i\neq p_j).$ Therefore, by Lemma \ref{lemma:any_odtwo_commuting_elements_of_prime_order_adjacent}, $a'\sim b'.$ As a result, we get a path $a\sim a'\sim b'\sim b.$ 
	That is, there exists a path of length $3$ between $a$ and $b.$ 
	We observe that this case takes care of everything except 
	when both $\text{o}(a)$ and $\text{o}(b)$ are power of the same prime $p_{\ell}$ for some $\ell \in [r].$ We next consider this case: 
	
	Case 2: Let $\text{o}(a)=p_{\ell}^{k_{\ell}}$ and $\text{o}(b)=p_{\ell}^{s_{\ell}}.$  Let, $c$ 
	be an element of order $p_i$ in $G_1$ with $i \neq \ell$ (as $r\geq 2).$ Then by Lemma \ref{lemma:any_odtwo_commuting_elements_of_prime_order_adjacent},
	we have $a \sim c \sim b$. 
	Thus, $\mathcal{G}_E^*(G_1)$ is connected. This completes the proof.
\end{proof}

\begin{corollary}\label{rem:diam_leq_3_when_nonpgroup}
	Let $G_1$ be a nilpotent non $p$-group having no sylow subgroups which are either cyclic or generalized quaternion. Then, $\text{diam}(\mathcal{G}_E^{**}(G_1)) \leq 3.$ 
\end{corollary}
%We know that any finite abelian group is nilpotent. Thus we have an immediate corollary.
%%\begin{corollary}[Theorem 1.2, \cite{bera-dey-mukherjee-connectivity-enhanced}]\label{ enhcd vc 1 iff g is p-group}
%%	Let $G$ be a finite non-cyclic abelian group. Then $\kappa(\mathcal{G}_E(G))$ is equal to $1$ if and only if $G$ is a $p$-group.  
%%\end{corollary}
%\begin{corollary}\label{cor: proper enh of G nilpotent having  no   cyclic and  quaternion disconnected iff} 
%	Let $G_1$ be a nilpotent group having no sylow subgroups which are either cyclic or generalized quaternion. Then $\mathcal{G}^{*}_E(G_1)$ is disconnected if and only if $G_1$ is $p$-group.	
%\end{corollary}
%\begin{proof}
%	Proof of this corollary follows from Theorem \ref{thm: nilpotent no sylow subgroups cyclic quaternion vc=1 iff}.
%\end{proof}
\begin{theorem}\label{iff thm:vc=n for G times cyclic group}
Let $G_1$ be a finite nilpotent group having no sylow subgroups which are either cyclic or generalized quaternion. Let for $n\in \mathbb{N}, \text{gcd}(|G_1|, n)=1.$ Then $\kappa(\mathcal{G}_E(G_1\times\mathbb{Z}_n))=n$ if and only if $G_1$ is $p$-group.
\end{theorem}
\begin{proof}
Let $G_1$ be a finite $p$-group such that $G_1$ is neither cyclic nor generalized quaternion group and $\text{gcd}(|G_1|, n)=1.$ We show that $\kappa(\mathcal{G}_E(G_1\times\mathbb{Z}_n))=n.$ Let $D=\{(e, x): e \text{ is the identity of } G_1 \text{ and } x\in \mathbb{Z}_n\}.$ Then by Proposition \ref{Prop: all dom of G times Zn}, $\text{Dom}(\mathcal{G}_E(G_1\times\mathbb{Z}_n))=D.$	
Now clearly to disconnect the graph we have to delete all vertices in $D.$ So, $\kappa(\mathcal{G}_E(G_1\times\mathbb{Z}_n))\geq n.$ 
Let $(g', e'), (g'', e')\in V(\mathcal{G}^{**}_E(G_1\times\mathbb{Z}_n))$ such that $\text{o}((g', e'))=\text{o}((g'', e'))=p$ and $\langle (g', e')\rangle\neq \langle (g'', e')\rangle.$	
Now we show that there is no path between the vertices $(g', e')$ and $(g'', e')$ in $\mathcal{G}^{**}_E(G_1\times\mathbb{Z}_n).$ Let $\mathcal{P}$ be a path $v_0\sim v_1\sim\cdots\sim v_{k-1}\sim v_k$ in $\mathcal{G}^{**}_E(G_1\times\mathbb{Z}_n),$ where $v_0=(g', e')$ and $v_k=(g'', e').$ Note that $v_0\sim v_1.$ So, there exists $v_1'\in G_1\times\mathbb{Z}_n$ such that $v_0, v_1\in \langle v_1'\rangle.$ Now, $v_0\in \langle v_1'\rangle$ and $\langle v_1'\rangle$ is cyclic. So, $\langle v_0\rangle$ is the unique cyclic subgroup of order $p$ of $\langle v_1'\rangle.$ Again, $v_1\in \langle v_1'\rangle$ and $p$ divides $\text{o}(v_1)$ (Note that $p$ divides the order of each vertex in the graph $\mathcal{G}^{**}_E(G_1\times\mathbb{Z}_n).$ Therefore, $p$ divides the order of each vertex in the path $\mathcal{P}).$ As a result, $v_0\in \langle v_1\rangle.$    
Again, $v_1\sim v_2$ and $v_0\in \langle v_1\rangle$ (already we have shown). Now, proceeding exactly same way as above, it can be shown that $v_0\in \langle v_2\rangle.$ Continuing this way, we  show that $v_0\in \langle v_{k-1}\rangle.$ Again, $v_{k-1}\sim v_k.$ In a similar manner, it can be shown that $v_k\in \langle v_{k-1}\rangle.$ As a result $\langle v_0\rangle=\langle v_k\rangle$ (as $\langle v_{k-1}\rangle$ is cyclic and it has a unique subgroup of order $p)$. But first we choose $v_0, v_k$ such that $\langle
v_0\rangle\neq \langle v_k\rangle.$ which is a contradiction. Hence the graph $\mathcal{G}^{**}_E(G_1\times\mathbb{Z}_n)$ is disconnected and $\kappa(\mathcal{G}^{**}_E(G_1\times\mathbb{Z}_n))=n.$ 
	
Conversely, let $|G_1|=p_1^{\alpha_1}\times\cdots\times p_t^{\alpha_t},$ where $\alpha_i\geq 1, t\geq 2$ and $p_1, \cdots, p_t$ are primes such that $p_i\neq p_j,$ for $i\neq j.$ Here we show that the graph $\mathcal{G}^{**}_E(G_1\times\mathbb{Z}_n)$ is connected. Let $v\in V(\mathcal{G}^{**}_E(G_1\times\mathbb{Z}_n))$ and we consider the set \[\text{Div}(v)=\{p_i: p_i \text{ is a prime divisor of both  }\text{o}(v) \text{ and } |G_1|\}.\] Clearly for each $p_i\in \text{Div}(v), p_i$ does not divide $n$ (as $\text{gcd}(|G_1|, n)=1).$  Let $v', v''$ be two arbitrary vertices of $\mathcal{G}^{**}_E(G_1\times\mathbb{Z}_n).$ We split the proof of this part into two cases.  
	
Case 1: First suppose that at least one of $\text{Div}(v'), \text{Div}(v'')$ has cardinality greater than equal to $2.$ With out loss of generality we assume that $|\text{Div}(v')|\geq 2.$ So we can choose $p_{i}\in \text{Div}(v')$ and $p_{j}\in \text{Div}(v'')$ such that $p_i\neq p_j.$ Let $v_1, v_2\in V(\mathcal{G}^{**}_E(G_1\times\mathbb{Z}_n))$ such that $\text{o}(v_1)=p_i, \text{o}(v_2)=p_j$ and $v_1\in \langle v'\rangle, v_2\in \langle v''\rangle.$ As a result $v'
\sim v_1$ and $v''\sim v_2.$ Again, $G_1\times \mathbb{Z}_n$ is nilpotent, so $v_1v_2=v_2v_1$ and $\text{gcd}(p_i, p_j)=1.$ Therefore by Lemma \ref{lemma:any_odtwo_commuting_elements_of_prime_order_adjacent}, $v_1\sim v_2$ in $\mathcal{G}^{**}_E(G_1\times\mathbb{Z}_n).$ Consequently, we get a path $v'\sim v_1\sim v_2\sim v''$ in $\mathcal{G}^{**}_E(G_1\times\mathbb{Z}_n).$
	
Case 2: Let $|\text{Div}(v')|=1$ and $ |\text{Div}(v'')|=1.$ 
Suppose that $p_i\in \text{Div}(v') \cap  \text{Div}(v'').$  This implies that $\text{o}(v')=p_i^{\ell_i}k_1$ and 
$\text{o}(v'')=p_i^{s_i}k_2$ for some positive integers $k_1, k_2$ dividing $n.$
Since $G_1$ is non $p$-group, then there exists $v_5 \in V(\mathcal{G}^{**}_E(G_1\times\mathbb{Z}_n))$ such that $\text{o}(v_5)=p_j$ with $p_i\neq p_j$ and $\text{gcd}(p_j,n)=1.$ Thus, $\text{gcd}(p_j, \text{o}(v')) =1 = 
\text{gcd}(p_j, \text{o}(v'')).$ As $G_1$ is nilpotent, any two elements of coprime order commute. Hence, the elements $v'$ and $v_5$ commute in 
$\mathcal{G}_E^{**}(G_1 \times \mathbb{Z}_n).$ Similarly, $v''$ and $v_5$ also commute. 
% So there exist two vertices, say $v_3, v_4\in V(\mathcal{G}^{**}_E(G_1\times\mathbb{Z}_n))$ such that $\text{o}(v_3)=p_i=\text{o}(v_4)$ and $v_3\in \langle v'\rangle, v_4\in \langle v''\rangle.$ Since $G_1$ is non $p$-group, then there exists $v_5\in V(\mathcal{G}^{**}_E(G_1\times\mathbb{Z}_n))$ such that $\text{o}(v_5)=p_j$ with $p_i\neq p_j.$ Now $G_1\times \mathbb{Z}_n $ is nilpotent. Therefore, $v_3v_5=v_5v_3$ and $v_4v_5=v_5v_4.$ Moreover, $\text{gcd}(\text{o}(v_3), \text{o}(v_5))=1$ and $\text{gcd}(\text{o}(v_4), \text{o}(v_5))=1.$ 
 So by Lemma \ref{lemma:any_odtwo_commuting_elements_of_prime_order_adjacent}, $v' \sim v_5 \sim v''$ 
 in $\mathcal{G}^{**}_E(G_1\times\mathbb{Z}_n).$ Consequently, $\mathcal{G}^{**}_E(G_1\times\mathbb{Z}_n)$ is connected and hence $\kappa(\mathcal{G}_E(G_1\times\mathbb{Z}_n))>n.$          \end{proof}
\begin{corollary}\label{cor: Proper enh of G times cyclic is disconnected iff }
Let $G_1$ be a finite nilpotent group having no sylow subgroups which are either cyclic or generalized quaternion. Suppose for $n\in \mathbb{N},  \text{gcd}(|G_1|, n)=1.$ Then $\mathcal{G}^{**}_E(G_1\times\mathbb{Z}_n)$ disconnected if and only if $G_1$ is $p$-group.
\end{corollary}
\begin{proof}
By Proposition \ref{Prop: all dom of G times Zn}, $\text{Dom}(\mathcal{G}_E(G_1\times \mathbb{Z}_n))=\{(e, x): e \text{ is the identity of } G_1 \text{ and } x\in \mathbb{Z}_n\}.$ Clearly, $|\text{Dom}(\mathcal{G}_E(G_1\times \mathbb{Z}_n))|=n.$ Again, by Theorem \ref{iff thm:vc=n for G times cyclic group}, $\kappa(\mathcal{G}_E(G_1\times \mathbb{Z}_n))=n$ if and only if $G_1$ is $p$-group. i.e., if $G_1$ is non $p$-group then the proper enhanced power graph $\mathcal{G}^{**}_E(G_1\times \mathbb{Z}_n)$ would be connected. Thus $\mathcal{G}^{**}_E(G_1\times\mathbb{Z}_n)$ is disconnected if and only if $G_1$ is $p$-group. 	
\end{proof}
\begin{proof}[Proof of Theorem \ref{Thm:no sylow subgroups quaternion, ** connected iff}]Proof of this theorem
follows from Theorem \ref{thm: nilpotent no sylow subgroups cyclic quaternion vc=1 iff} and Corollary \ref{cor: Proper enh of G times cyclic is disconnected iff }.
\end{proof}
\begin{corollary}\label{rem:diam_leq_3_when_nonpgroup_timescyclic}
Let $G_1$ be a nilpotent non $p$-group having no sylow subgroups which are either cyclic or generalized quaternion. Then $\text{diam}(\mathcal{G}_E^{**}(G_1 \times \mathbb{Z}_n)) \leq 3.$ 
\end{corollary}
\begin{theorem}\label{ thm:graph ** connected for G times quaternon group}
Let $G_1$ be a finite nilpotent group having no sylow subgroups which are either cyclic or generalized quaternion. Let $\text{gcd}(|G_1|, 2)=1.$ Then $\mathcal{G}^{**}_E(G_1\times Q_{2^k})$ is connected. 
\end{theorem}
\begin{proof}
By Proposition \ref{prop: dom; nilpotent grp times quaternion grp}, $D=\text{Dom}(\mathcal{G}_E(G_1\times Q_{2^k}))=\{(e, e''), (e, y)\},$ where $\text{o}(y)=2.$ Let $v_1, v_2$ be two arbitrary vertices in the graph $\mathcal{G}_E^{**}(G_1\times Q_{2^k})$ such that $\text{o}(v_1)=m_1$ and $\text{o}(v_2)=m_2.$ Now we have the following cases:

Case 1: Let $\text{gcd}(m_1, 2)=1=\text{gcd}(m_2, 2).$ We choose $v_3\in V(\mathcal{G}_E^{**}(G_1\times Q_{2^k}))$ such that $\text{o}(v_3)=2^t, t\geq 2.$ Then  $v_1v_3=v_3v_1$ and $v_2v_3=v_3v_2$ (as $G_1\times Q_{2^k}$ is nilpotent). So, by Lemma \ref{lemma:any_odtwo_commuting_elements_of_prime_order_adjacent} we get a path $v_1\sim v_3\sim v_2$ in $\mathcal{G}_E^{**}(G_1\times Q_{2^k}).$ 

Case 2: Let $\text{gcd}(m_1, |G_1|)=1=\text{gcd}(m_2, |G_1|).$ Similarly as Case 1, we get a path $v_1\sim v_4\sim v_2$ in $\mathcal{G}_E^{**}(G_1\times Q_{2^k}),$ where $\text{o}(v_4)$ divides $|G_1|$ and $\text{gcd}(\text{o}(v_4), 2)=1.$ 

Case 3:  Let $\text{gcd}(m_1, |G_1|)\neq 1$ and $ \text{gcd}(m_1, 2)\neq 1.$ So, there exists a prime $p$ (with $\text{gcd}(p, 2)=1)$ such that $p$ divides both $|G_1|$ and $m_1.$ As a result, we get a vertex say $v_5\in V(\mathcal{G}_E^{**}(G_1\times Q_{2^k}))$ such that $\text{o}(v_5)=p$ and $v_5\in \langle v_1\rangle.$ Now we have the following choices; 
\begin{enumerate}
	\item
	either $\text{gcd}(m_2, |G_1|)\neq 1$ and $ \text{gcd}(m_2, 2)\neq 1$
	\item
	or $\text{gcd}(m_2, |G_1|)\neq 1$ and $\text{gcd}(m_2, 2)=1$
	\item
	or $\text{gcd}(m_2, 2)\neq 1$ and $\text{gcd}(m_2, |G_1|)=1.$ 	 
\end{enumerate}
For the first choice, there exists a prime $p'$ with $\text{gcd}(p', 2)=1$ such that $p'$ divides both $|G_1|$ and $m_2, (\text{ it is possible as } \text{gcd}(|G_1|, 2)=1).$ As a result, we get a vertex $v_6$ (say) in $V(\mathcal{G}_E^{**}(G_1\times Q_{2^k}))$ such that $\text{o}(v_6)=p'$ and $v_6\in \langle v_2\rangle.$ Now we choose a vertex say $v_7\in V(\mathcal{G}_E^{**}(G_1\times Q_{2^k}))$ such that $\text{o}(v_7)=2^t, t\geq 2.$ Then using Lemma \ref{lemma:any_odtwo_commuting_elements_of_prime_order_adjacent}, we get a path $v_1\sim v_5\sim v_7\sim v_6\sim v_2$ (as $v_7v_5=v_5v_7$ and $v_7v_6=v_6v_7).$ Similarly, for the other choices we can find a path in $\mathcal{G}_E^{**}(G_1\times Q_{2^k})$ between $v_1$ and $v_2.$ This completes the proof.  
\end{proof}
\begin{corollary}\label{rem:diam_leq_3_when_nonpgroup_timescyclic_quarternion}
Let $G_1$ be a nilpotent group having no sylow subgroups which are either cyclic or generalized quaternion. Then, $\text{diam}(\mathcal{G}_E^{**}(G_1 \times Q_{2^k} )) \leq 4.$ 
\end{corollary}
\begin{theorem}\label{Thm: ** connected when G times Cyclic times quaternion}
	Let $G_1$ be a nilpotent group having no sylow subgroups which are either cyclic or generalized quaternion. If $n\in \mathbb{N}, \text{ and gcd}(|G_1|, n)=\text{gcd}(|G_1|, 2)=\text{gcd}(n, 2)=1,$ then $\mathcal{G}^{**}_E({G_1\times\mathbb{Z}_n\times Q_{2^k}})$ is connected. Here also,  $\text{diam}(\mathcal{G}_E^{**}(G_1 \times\mathbb{Z}_n\times Q_{2^k} )) \leq 4.$ 
\end{theorem}
\begin{proof}
	Proceeding exactly same way as in the proof of Theorem \ref{ thm:graph ** connected for G times quaternon group}, it can be shown that the graph $\mathcal{G}^{**}_E({G_1\times\mathbb{Z}_n\times Q_{2^k}})$ is connected and  $\text{diam}(\mathcal{G}_E^{**}(G_1 \times\mathbb{Z}_n\times Q_{2^k} )) \leq 4.$  	
\end{proof} 
\begin{proof}
[Proof of Theorem \ref{Thm: having sylow subgroups quaternion, ** connected }] Follows from Theorem \ref{ thm:graph ** connected for G times quaternon group} and Theorem 
\ref{Thm: ** connected when G times Cyclic times quaternion}. 
\end{proof}
In Corollaries \ref{rem:diam_leq_3_when_nonpgroup}, \ref{rem:diam_leq_3_when_nonpgroup_timescyclic} and \ref{rem:diam_leq_3_when_nonpgroup_timescyclic_quarternion} and Theorem \ref{Thm: ** connected when G times Cyclic times quaternion}, we have seen that the diameter of the proper enhanced power graphs of finite nilpotent groups are always $\leq 4.$ In the following result we show that when $G$ does not have any sylow subgroup that is generalized quarternion, the diameter is always exactly $3.$
\begin{theorem}
  	\label{thm:diam_prop_enhcd_pwr_abelian}
  	Let $G$ be a noncyclic nilpotent group without  any sylow subgroup which is generalized quarternion and 
  	$\mathcal{G}_E^{**}(G)$ is connected. Then $ \text{diam}(\mathcal{G}_E^{**}(G))=3$ . 
  \end{theorem}
  
  \begin{proof}
  	From Theorem \ref{Thm:no sylow subgroups quaternion, ** connected iff}, we can see that the proper enhanced power graph $\mathcal{G}_E^{**}(G)$ of a finite  group $G$ is connected if and only if 
  	\begin{enumerate}
  		\item  either $G=G_1,$ where $G_1=P_1 \times  P_2 \times  \dots \times  P_m,$ $m \geq 2$ and each $P_i$ is neither cyclic nor generalized quarternion,
  		\item or $G =   G_1 \times \mathbb{Z}_n,$ where $G_1$ is as (1) and $\text{gcd}(|G_1|, n) =1$. 
  	\end{enumerate}
  	
  	We at first show that for any such group $G$, the diameter of $\mathcal{G}_E^{**}(G)$ is $\leq 3.$ This directly follows from 
  	Corollaries \ref{rem:diam_leq_3_when_nonpgroup} and \ref{rem:diam_leq_3_when_nonpgroup_timescyclic}. 
  	We now show that for any such group $G$, the diameter of $\mathcal{G}_E^{**}(G)$ is $\neq 2$. At first we assume that
  	$G= P_1 \times P_2 \times  \dots \times
  	P_m,$ $m \geq 2. $
  	We produce two elements $x,y \in G$ such that $d(x,y)>2$. 
  	Let $x= (x_1,x_2,\dots,x_m)$ and $y=(y_1,y_2,\dots,y_m),$ where  for each $i \in [m]$,  we choose $x_i (\neq e_{P_i})$ and $y_i (\neq e_{P_i} )$ such that $x_i \nsim y_i   $. Clearly $d(x,y) \neq 1$. This is possible, since for each 
  	$i \in [m],$ $P_i$ is a non-cyclic $p$-group. If possible, let the distance of $x$ and $y$ in the proper enhanced power graph $\mathcal{G}_E^{**}(G)$ is $2$. Then, there exists an element, say $z= (z_1, z_2, \dots, z_m)$ such that $x \sim z \sim y$ and $z \in \mathcal{G}_E^{**}(G)$. Now $x \sim z $ implies that there exists some $u= (u_1, u_2, \dots, u_m)$ such that both $x$ and $z$ are multiple of $u$. Therefore $x_1 \sim z_1$ in the proper enhanced power graph $\mathcal{G}_E^{**}(P_1)$. 
  	In a similar way, $z_1 \sim y_1$ in the proper enhanced power graph $\mathcal{G}_E^{**}(P_1)$ and this forces $z_1= e_{P_1}$. Similarly, we can show that $z_i=e_{P_i}$ for each $ i \in [m]$. Hence $z= e,$ the identity of $G$ and therefore $z  \notin  \mathcal{G}_E^{**}(G)$, contradiction. 
  	
  	We now move on to the case when $G= G_1 \times \mathbb{Z}_n,$ where $G_1 = P_1 \times P_2 \times  \dots \times P_m$ and $\text{gcd}(p_1p_2\dots p_m, n)=1$. Here also, our intention is same, that is,
  	to produce two elements $x,y \in G$ such that $d(x,y)>2$. 
  	Let $x= (x_1,x_2,\dots,x_m, e')$ and $y=(y_1,y_2,\dots,y_m, e')$ where for each $i \in [m],$  we choose $x_i (\neq e_{P_i})$ and $y_i (\neq e_{P_i} )$ such that $x_i \nsim y_i   $. Let $z=(z_1, z_2, \dots, z_m, a)$ be an element such that $x \sim z \sim y$ and $z \in \mathcal{G}_E^{**}(G)$. By proceeding similarly as in above, we can show that $z= (e_{G_{p_1}},e_{G_{p_2}},\dots,e_{G_{p_m}},a)=(e_{G_1},a) $ for some $a \in \mathbb{Z}_n$, and therefore   $z  \notin  \mathcal{G}_E^{**}(G)$, contradiction. 
  	
  	Thus, for any finite nilpotent group $G$ such that $\mathcal{G}_E^{**}(G)$ is connected and it does not have any sylow subgroup which is generalized quarternion, we have found two elements whose distance is $\geq 3$ and therefore 
  	diameter of $\mathcal{G}_E^{**}(G)$ is $\leq 3$, completing the proof. 
  \end{proof}

 \section{Domination number of proper enhanced power graph}
 \label{sec:domination_number_proper}

  In this section, we determine the domination number and diameter of the graph $\mathcal{G}_E^{**}(G)$ for any finite nilpotent group $G$. For this purpose, we start with counting the number of components of the proper enhanced power graph $\mathcal{G}_E^{**}(G)$.

 \begin{theorem}
 	\label{thm:components_of_G_abelian}
 	Let G be a finite $p$-group which is neither cyclic nor generalized quarternion.
 	Then, the number of components of $\mathcal{G}_{E}^{**} (G )$  
 	is same as the $\text{number of distinct p-order subgroups of } G.$ 
 \end{theorem}
 
 \begin{proof}
 	Let $H_1$, $H_2, \dots, $ $H_s$ be the distinct $p$-order subgroups of $G$. We claim that $H_1 \cup N(H_1) $, $H_2 \cup N(H_2)$, \dots, $H_s \cup N(H_s)$ are disjoint components of $\mathcal{G}_{E}^{**} (G )$. Consider any non-identity element $a$ of $G.$ Clearly, there exists $r$ such that $a^r$ is of order $p$. As $a^r \in \langle a \rangle,$ $a$ and $a^r$ are adjacent in $\mathcal{G}_{E}^{**} (G )$. So, $a$ must be in one of the components $H_1 \cup N(H_1) $, $H_2 \cup N(H_2)$, \dots, $H_s \cup N(H_s).$ Thus, the number of components is $\leq s.$  By Lemma \ref{lema: p and p^i orderd path connd abelong< b>}, 
 	if any two elements of order $p$ are connected by a path, 
 	then one of them must be the multiple of another. Henceforth, there are at least $s$ many components. 
 	This completes the proof. 
 \end{proof}

 \begin{theorem}
 	\label{thm:no_of_components_nilpotent}
 	Let $G$ be a nilpotent group such that  $G \cong G_1 \times \mathbb{Z}_n,$ where $G_1$ is a finite $p$-group which is neither cyclic nor generalized quarternion. Then, the number of components of  $\mathcal{G}_{E}^{**} (G )$ is same as the $\text{number of distinct p-order subgroups of } G_1.$ 
 \end{theorem}

 \begin{proof}
  By Theorem \ref{thm:components_of_G_abelian}, we see that the number of connected components of  $\mathcal{G}_{E}^{**} (G_1)$ is $s$ where $s$ is the number of distinct $p$-order subgroups of $G.$ Let, $C_i  = H_i \cup N(H_i) $. Then by Theorem \ref{thm:components_of_G_abelian}, $C_1, C_2, \cdots, C_s$ are the components of 
 	$C(\mathcal{G}_{E}^{**} (G_1))$. Define $f: C(\mathcal{G}_{E}^{**} (G_1)) \mapsto C(\mathcal{G}_{E}^{**} (G_1 \times \mathbb{Z}_n))$ by
 	$$f(C_i)= C_i \times  \mathbb{Z}_n .$$
 	
 	Clearly, the number of components is at most $s$. Thus it is enough to show that there is no path in between $C_i \times  \mathbb{Z}_n$ and 
 	$C_j \times  \mathbb{Z}_n$ for $1 \leq i < j \leq s.$ This follows in an identical manner to the proof of Theorem \ref{thm:at_least_r_components}.  Therefore, the number of components of $\mathcal{G}_{E}^{**} (G)$ is at least 
 	$s.$ The proof is complete. 
 \end{proof}

 \begin{proof}[Proof of Theorem \ref{thm:domination number p-group times cyclic_nilpotent}]
 	Let $H_1$, $H_2, \dots, $ $H_s$ be the distinct $p$-order subgroups of $G_1$. From the proof of Theorem \ref{thm:components_of_G_abelian}, we see that $H_1 \cup N(H_1) $, $H_2 \cup N(H_2)$, \dots, $H_s \cup N(H_s)$ are disjoint components of $\mathcal{G}_{E}^{**} (G ).$ Thus,  the domination number of $G_1$ is clearly $\geq s$. 
 	For $ 1 \leq i \leq s$, let $a_i$ be some element of order $p$ which is chosen from  $H_i $. 
 	Consider the following set
 	$$D_1= \{a_1,a_2, \dots, a_{s}\}.$$
 	By Lemma \ref{lema: p and p^i orderd path connd abelong< b>}, the component $C_i$ is dominated by the element $a_i$ and therefore $D_1$ is a dominating set for $\mathcal{G}_E^{**}(G_1).$ This completes the proof for $G_1.$ 
 	
 	We now consider the case when $G= G_1 \times \mathbb{Z}_n.$ By Theorem \ref{thm:no_of_components_nilpotent}, the number of components of $\mathcal{G}_E^{**}(G)$ is at least $s.$ Consider the following set
 		$$D_2= \{(a_1, e'), (a_2, e')), \dots, (a_s, e') \}$$ where $e'$ denotes the identity element of $\mathbb{Z}_n$ and $a_i \in H_i.$
 		We claim that the element $(a_i, e')$ dominates the component $C_i \times \mathbb{Z}_n$. Let $(x,y) \in C_i \times \mathbb{Z}_n$. Then, $(x,y)^n=(x^n,  e') $ and as $\text{gcd}(n,p)=1$, we have $a_i = {x^{n}}^r$ for some $r \in \mathbb{N}$. Therefore, we have $(a_i, e')=(x,y)^{rn}$ and this proves our claim. Hence, the set $D_2$ dominates  $\mathcal{G}_{E}^{**} (G)$ and the proof is complete. 	 
 \end{proof}
 
 For an abelian group $G$ of order $p^r,$ the number of distinct $p$-order subgroups is $\frac{p^r-1}{p-1}.$ So, we immediately get the following. 

 \begin{theorem}
 	\label{thm:domination number p-group times cyclic}
 	Let $G_1$ be a finite abelian noncyclic $p$-group. Suppose that
 	$$G_1 = \mathbb{Z}_{p^{t_1}}
 	\times  \mathbb{Z}_{p^{t_2}}
 	\times  \cdots \times 
 	\mathbb{Z}_{p^{t_r}}.$$ where $r \geq 2$ and $1 \leq t_1 \leq t_2 \leq \cdots \leq t_r.$
 	In this case,  $\gamma(\mathcal{G}^{**}_e(G_1))=\frac{p^r-1}{p-1}.$ 
 	Let $G$ be an abelian group such that  $G = G_1 \times \mathbb{Z}_n,$ where $r \geq 2$ and $\text{gcd}(p, n)=1.$ Then also we have  $\gamma(\mathcal{G}_{E}^{**} (G )= \frac{p^r-1}{p-1}.$ 
 \end{theorem}

 \begin{proof}[Proof of Theorem \ref{domination number of enhced pwr raph for connected_nilpotent}]
 	At first, we consider the case when	$G_1$  is of the following form: 
 	\[G_1 = P_1 \times P_2 \times \dots \times  P_m \] where $m \geq 2$ and for each $ i \in [m],$ $P_i$ a $p_i$-group which is neither cyclic nor generalized quarternion.    
 	As in the proof of Theorem \ref{thm:domination number p-group times cyclic}, consider $\mathcal{G}_E^{**}(P_1)$
 	and for $ 1 \leq i \leq s_1$, let $a_i$ be some element of order $p_1$ which is chosen from the component $C_i$. 
 	Consider the following set
 	$D_1= \{a_1,a_2, \dots, a_{s_1}\}$
 	which is a dominating set of $\mathcal{G}_E^{**}(P_1).$ Let $D_3=\{ (d,e_{P_2}, \dots, e_{P_m} ): d \in D_1 \}$ where $e_{P_i}$ denotes the identity of $P_i$.  Let $(x_1, x_2, \dots, x_m) \in \mathcal{G}_E^{**}(G_{1})$. Then $(x_1, x_2, \dots, x_m)^{p_2^{t_{2k_2}} p_3^{t_{3k_3}} \dots p_m^{t_{mk_m}}}= (x_1',e_{P_2}, \dots, e_{P_m})$ for some $x_1' \in P_1$. As $D_1$ is a dominating set of $\mathcal{G}_E^{**}(P_1)$, by Lemma \ref{lema: p and p^i orderd path connd abelong< b>} there exists $d'$ such that $d'=(x_1')^{\ell}$ and hence we have 
 	$$(d',e_{P_2}, \dots, e_{P_m})) \sim (x_1, x_2, \dots, x_m).  $$
 	Thus $D_1$ is a dominating set of $\mathcal{G}_E^{**}(G_{1})$ and therefore $$\gamma(\mathcal{G}_E^{**}(G_1)) 
 	\leq   s_1 .$$
 	We can similarly show that $\gamma(\mathcal{G}_e^{**}(G_1)) 
 	\leq   s_i $ for any $ 1 \leq i \leq m$ and that proves  $$\gamma(\mathcal{G}_E^{**}(G_1)) \leq 
 	\displaystyle \min _{ 1 \leq i \leq m}  s_i .$$ 
 		
 	Let $D$ be a dominating set of $\gamma(\mathcal{G}_E^{**}(G_1))$ of cardinality $\ell < \min _{ 1 \leq i \leq m}  s_i  $. Let  
 	$D= \{ 
 	(w_{11},w_{12},\dots, w_{1m}), \dots, (w_{\ell 1},w_{\ell 2},\dots,w_{\ell m}) \}.$ Therefore, $ |D| < |D_i|$ for each $i$ with $1 \leq i \leq m$, where $D_i$ is a dominating set of minimum cardinality for $\mathcal{G}_E^{**}(P_i).$ Hence, for each $ i \in [m],$ there exists some $y_i \in P_i$ such that $y_i$ is not dominated by any of the vertices among  $\{w_{1i},w_{2i},\dots,w_{\ell i}\} $ in the graph $\mathcal{G}_E^{**}(P_i).$ Consider the vertex $u=(y_1,y_2,\dots,y_m)$. If $u$ is dominated by some vertex of $D$ say $v=(v_1, v_2, \dots, v_m)$, then we must have $v_1= e_{P_1}, v_2= e_{P_2}, \dots, v_m= e_{P_m}$ and hence $v= e_{G_1}  $ which contradicts the fact that $v \in  \mathcal{G}_E^{**}(G_1).$ Hence, any dominating set of 
 	$\mathcal{G}_E^{**}(G_1)$  must have cardinality $ \geq \min _{ 1 \leq i \leq m}  s_i  .$ This proves \eqref{eqn:dominating_connected}.
 	
 	We now move on to the case when 
 	\[G= P_1 \times P_2 \times \dots \times  P_m \times \mathbb{Z}_n \] and we show that
 	$S=\{(d,e_{P_2}, \dots, e_{P_m},e_{\mathbb{Z}_n}): d \in D_1 \}$ is a dominating set of $\mathcal{G}_E^{**}(G).$ Let $(x_1, x_2, \dots, x_m,x_{m+1}) \in \mathcal{G}_E^{**}(G).$ Then \[(x_1, x_2, \dots, x_m,x_{m+1})^{p_2^{t_{2k_2}} p_3^{t_{3k_3}} \dots p_r^{t_{rk_r}}n}= (x_1',e_{{P_2}}, \dots, e_{{P_m}},e_{\mathbb{Z}_n})\] for some $x_1' \in G_{p_1}$. Hence, there exists $d'$ such that $d'=x_1'^l$ and hence we have 
 	$$(d',e_{{P_2}}, \dots, e_{{P_m}},e_{\mathbb{Z}_n}) \sim (x_1, x_2, \dots, x_m,x_{m+1}).  $$
 	Thus $D_1$ is a dominating set of $\mathcal{G}_E^{**}(G)$ and therefore $\gamma(\mathcal{G}_E^{**}(G_1)) 
 	\leq   s_1 .$
 	We can similarly show that $\gamma(\mathcal{G}_E^{**}(G_1)) 
 	\leq   s_i $ for any $i$ and that proves that $\gamma(\mathcal{G}_E^{**}(G)) \leq 
 	\displaystyle \min _{ 1 \leq i \leq m}  s_i .$
 	
 	Finally we are left to show that any dominating set  of  $\mathcal{G}_E^{**}(G))$ has cardinality 
 	$\geq 
 	\displaystyle \min _{ 1 \leq i \leq m}  s_i .$ This follows in an identical manner to the proof of the above fact that any dominating set of $\mathcal{G}_E^{**}(G_1))$ has cardinality 
 	$\geq 
 	\displaystyle \min _{ 1 \leq i \leq r}  s_i $ and therefore we omit this. Hence, \eqref{eqn:dominating_connected_2} is proved. 
 \end{proof}

 \begin{corollary}\label{domination number of enhced pwr raph for connected}	
 	Let $G_1$ be a product of non-cyclic abelian $p$-groups, that is, of the following form: \[G_1 = P_1 \times P_2 \times \dots \times  P_m\] where $m \geq 2$ and for all $ 1 \leq i \leq m$, we have \[ P_i = {Z}_{p^{t_{i1}}_i} \times \mathbb{Z}_{p^{t_{i2}}_i} \times \dots \mathbb{Z}_{p^{t_{ik_i}}_i}\]
 	%				\times \mathbb{Z}_{p^{t_{21}}_2}\times\mathbb{Z}_{p^{t_{22}}_2}\times\cdots\times\mathbb{Z}_{p^{t_{2k_2}}_2}\times\cdots 
 	with $k_i \geq 2$ and  $1\leq t_{i1}\leq t_{i2}\leq\cdots\leq t_{ik_i} $. Then
 	\begin{equation*}
 	\label{eqn:dominating_connected3} \gamma(\mathcal{G}_E^{**}(G_1)) = 
 	\displaystyle \min _{ 1 \leq i \leq m}  \frac{ p_i^{k_i}-1}{p_i-1}. 
 	\end{equation*}
 	Let $G = G_1 \times \mathbb{Z}_n$ with $\text{gcd}(n,|G_1|)=1$, then also 
 	\begin{equation*}
 	\label{eqn:dominating_connected_4}
 	\gamma(\mathcal{G}_E^{**}(G)) = 
 	\displaystyle  \min _{ 1 \leq i \leq m}    \frac{ p_i^{k_i}-1}{p_i-1}.
 	\end{equation*}
 \end{corollary}

%We now move on to find the diameter of the proper enhanced power graph $\mathcal{G}_E^{**}(G)$.
% For that purpose, we start with the following definition. 

  %\section{Diameter of $\mathcal{G}_E^{**}(G)$ when $G$ is abelian}

%  
%    \begin{defn}
%    	Let $G$ be a finite group and $a\in G$. Define the {\it Support of }$a$, $\mathsf{Supp}(a)=\{p:p \mbox{ is a prime dividing order of }a \}$ and {\it Support of }$G$, $\mathsf{Supp}(G)=\{p:p \mbox{ is a prime dividing }|G| \}$.
%    \end{defn}

\section{Multiplicity of Laplacian spectral radius}
\label{sec:mul_sec_radius}

In this section, we find the multiplicity of the Laplacian spectral radius of the enhanced power graph of any finite nilpotent group. For that purpose, we recall the definition of the Laplacian matrix $L(\Gamma)=(L_{i,j})_{n \times n}$ of a graph $\Gamma$ with
the vertex set $\{ v_1, v_2, \dots, v_n \}$, where
\begin{equation*}
L_{i, j} = 
\begin{cases}
d_i, & \text{ if } i=j \\
-1, & \text{ if } i \neq j \text{ and } v_i \sim v_j     \\
0, & \text{ if } i \neq j \text{ and } v_i \nsim v_j 
\end{cases}
\end{equation*}
and $d_i$ is the degree of the vertex $v_i.$
 For any graph $\Gamma,$ the characteristic polynomial $\det(xI - L(\Gamma ))$ of $L(\Gamma)$ is called
the Laplacian characteristic polynomial of $\Gamma$ and is denoted by $\Theta ( \Gamma , x).$
Let $\lambda_1(\Gamma)$ be the largest eigenvalue of $L(\Gamma)$. Let 
$$\lambda_1(\Gamma) \geq \lambda_2(\Gamma) \geq \dots \geq \lambda_n(\Gamma) $$ be the eigenvalues of the Laplacian matrix $L(\Gamma)$. The highest eigenvalue $\lambda_1(\Gamma)$ is called the \emph{Laplacian spectral radius} of $\Gamma.$  We denote by $\eta(\lambda_1(\Gamma)),$ the multiplicity of the Laplacian spectral radius or the highest eigenvalue. For a graph $\Gamma,$ let $\Gamma^{\text{com}}$ denote the complement of the graph $\Gamma$ and 
$(\Gamma^{\text{com}})_*$ denote the induced subgraph on $\Gamma^{\text{com}}$ after removing the isolated vertices, if any. Dey in \cite{dey-laplacian_spectrum} proved the following result which connects the multiplicity of the Laplacian spectral radius of a graph $\Gamma.$ 

\begin{theorem}
	\label{thm:main}
	Let $\Gamma$ be a simple graph on $n(\geq 3)$ vertices. Then $ \eta (\lambda_1(\Gamma)) = |Dom(\Gamma)| $ if and only if $\Gamma$ is non-complete,  $({\Gamma}^{\text{com}})_*$ is connected and $\Gamma$ has at least one dominating vertex. 
\end{theorem}

For the sake of completeness, as the result is not yet published, we at first give a short proof of this result. For that we need the following result of Mohar \cite{mohar-laplacian_spectrum}. 

\begin{theorem}[Mohar]
	\label{thm:mohar_1} Let $\Gamma$ be a graph with $n$ vertices. Then $\lambda_ 1 ( \Gamma  ) \leq n.$ Equality
	holds if and only if $\Gamma^{\text{com}}$ is not connected.
\end{theorem}

The union of graphs $\Gamma_1$ and $\Gamma_2$, denoted by $\Gamma_1 \cup \Gamma_2$, is the graph with vertex set
$V (\Gamma_1 ) \cup V (\Gamma_2 )$ and edge set is the union of all the edges of $\Gamma_1$ and all the edges of $\Gamma_2.$ If $\Gamma_1$ and $\Gamma_2$
are disjoint, that is, they do not have any common vertices, we refer to their union as a disjoint
union, and denote it by
$\Gamma_1 + \Gamma_2$. If $\Gamma_1$ and $\Gamma_2$ are disjoint, their join 
$\Gamma_1 \vee \Gamma_2$ is the
graph obtained by taking $\Gamma_1 + \Gamma_2$ and adding all edges $\{u, v\}$ with $u \in V (\Gamma_1 )$ and
$v \in V (\Gamma_2 ).$ Mohar in the same paper \cite{mohar-laplacian_spectrum} proved the following result which provides the Laplacian spectrum of the join of two graphs. 

\begin{theorem}[Mohar] 
	\label{thm:mohar_2}
	Let $\Gamma_1$ and $\Gamma_2$ are disjoint graphs on $n_1$ and $n_2$ vertices. Then, 
	$$ \Theta (\Gamma_1 \vee \Gamma_2 , x) =  \frac{x(x-n_1-n_2)}{(x-n_1)(x-n_2)} \Theta (\Gamma_1, x-n_2) \Theta (\Gamma_2, x-n_1)  .$$
\end{theorem}

For a graph $\Gamma$ with $r$ dominating vertices, let $\Gamma_1$  be the complete graph $K_r$ by taking all the dominating vertices and let $\Gamma'$ be the induced graph on the remaining $n-r$ vertices. Clearly $\Gamma'$ is a graph in $n-r$ vertices with no dominating vertex.  By Theorem \ref{thm:mohar_2}, we then have,

\begin{eqnarray}
\Theta (K_r \vee \Gamma', x) & = & \frac{x(x-n)}{(x-r)(x-(n-r))} \Theta (K_r, x-(n-r)) \Theta (\Gamma', x-r) \nonumber \\
& = & \frac{x(x-n)}{(x-r)(x-n+r))} (x-n)^{r-1}(x-n+r) \Theta (\Gamma', x-r) \nonumber \\
& = & \frac{x(x-n)^r}{(x-r)} \Theta (\Gamma', x-r)  \label{eqn:vital}
\end{eqnarray}

The second line follows from the fact that the Laplacian matrix of the complete graph has eigenvalues $n$ with multiplicity $n-1$ and $0$ with multiplicity $1$. It is easy to note that $ \Theta (\Gamma', x-r)$ equals with the characteristic polynomial
of the submatrix of $L(\Gamma)$ obtained after deleting rows and columns corresponding to the dominating vertices of $\Gamma$. We then have 
\begin{equation}
\lambda_i(\Gamma)= n \hspace{ 3 mm}  \mbox{for} \hspace{3 mm}
{ 1 \leq i \leq r } \hspace{3 mm} \mbox{and} \hspace{3 mm} \lambda_{r+1}(\Gamma)=\lambda_1(\Gamma')+r 
\label{eqn:vital_2}
\end{equation}

\vspace{2 mm}

We are now in a position to prove Theorem \ref{thm:main}.

\vspace{2 mm}

\begin{proof}[Proof of Theorem \ref{thm:main}]
	We prove the forward implication at first. Let the multiplicity of the highest eigenvalue of $L(\Gamma)$ be the number of dominating vertices of $\Gamma$ and we need to show that $\Gamma$ is non-complete, has a dominating vertex and $({\Gamma}^{\text{com}})_*$ is connected. Firstly, the graph $\Gamma$ cannot be complete otherwise the number of dominating vertices would be $n$ and the multiplicity of the largest eigenvalue would be $n-1$. Secondly, $\Gamma$ must have a dominating vertex. Finally, if $({\Gamma}^{\text{com}})_*$ is disconnected then by Theorem \ref{thm:mohar_1}, the highest eigenvalue of $L(\Gamma')$ is $n-r$ as it is easy to check that $(\Gamma')^{\text{com}}=({\Gamma}^{\text{com}})_*$. Thus, by \eqref{eqn:vital_2}, we have $\lambda_{r+1} (\Gamma)= \lambda_1(\Gamma')+r=n-r+r=n$. Therefore the multiplicity of the highest eigenvalue is greater than the number of dominating vertices, which is a contradiction.
	
	For the other direction, we assume that $\Gamma$ is a simple, non-complete graph with a dominating vertex such that $({\Gamma}^{\text{com}})_*$ is connected. Assume that $\Gamma$ has $r$ dominating vertices and we form $K_r$ and $\Gamma'$ as above.  As $(\Gamma')^{\text{com}}=({\Gamma}^{\text{com}})_*$ and
	$({\Gamma}^{\text{com}})_*$  is connected, by Theorem \ref{thm:mohar_1}, we have $\lambda_1(\Gamma')<n-r$.  Thus
	$\lambda_{r+1} (\Gamma)= \lambda_1(\Gamma')+r < n-r+r=n$ and now by using \eqref{eqn:vital_2}  the proof is complete.  
\end{proof}	

Therefore, in order to find out the multiplicity of the Laplacian spectral radius of the enhanced power graph $\mathcal{G}_E(G)$ for any nilpotent group, we at first prove that $\mathcal{G}_E^{**}(G)^{\text{com}}$ is connected.

%\subsection{Connectivity of complement of $\mathcal{G}_E^{**}(G)$}
%Throughout this section we assume that $G_1$ is a nilpotent group having sylow subgroups which are neither cyclic nor generalized quaternion.

\begin{theorem}\label{thm: conncted complement graph of cyclic product G_1}
	Let $G$ be a noncyclic nilpotent group which does not have any cyclic sylow subgroup which is generalized quarternion.  Then $\mathcal{G}_E^{**}(G)^{\text{com}}$ is connected. 	
\end{theorem}
\begin{proof}
	Let $a,b$ be vertices of $ \mathcal{G}_E^{**}(G)^{\text{com}}.$ We claim that there exists a vertex $c$ such that $c \sim a $ in  $\mathcal{G}_E^{**}(G)^{\text{com}}$ and $c \sim b$ in $\mathcal{G}_E^{**}(G)^{\text{com}}.$
	Otherwise, every vertex is dominated by at least one of $a$ and $b$ in $\mathcal{G}_E^{**}(G)$ and hence $\gamma (\mathcal{G}_E^{**}(G)) \leq 2,$ which is a contradiction. Thus, our claim is true and hence there exists a path of length $2$ between $a$ and $b$ in $\mathcal{G}_E^{**}(G)^{\text{com}},$ completing the proof. 
\end{proof}
Let $G_1$ be a finite nilpotent group having no sylow subgroups which are generalized quarternion or cyclic. 
\begin{theorem}\label{Thm: mul spec rad for nilpotent grp}
Let $G$ be a finite nilpotent group having no sylow subgroup which is generalized quarternion. That is, either $G=G_1$ or $G=G_1\times \mathbb{Z}_n,$ where $\text{gcd}(|G_1|, n)=1.$ Then 
	\begin{equation*}
	 \eta (\lambda_1(\mathcal{G}_E(G))) = 
	\begin{cases}
	1, & \text{ if } G=G_1 \\
	n,& \text{ if } G=G_1\times \mathbb{Z}_n\text{ and } \text{gcd}(|G_1|, n)=1.
	\end{cases}
	\end{equation*}
\end{theorem}
\begin{proof}
By using Theorems \ref{thm:main} and \ref{thm: conncted complement graph of cyclic product G_1}, we can see that 
	$\eta(\lambda_1(\mathcal{G}_E(G)))= |\text{Dom}(\lambda_1(\mathcal{G}_E(G)))|$ when $G_1$ is a finite nilpotent group which has no sylow subgroup which is generalized quarternion. The proof of this theorem now follows from Theorem \ref{Thm: Dom set for nilpotent grp}. 	
\end{proof}
\subsection*{Acknowledgement}
The first author would like to thank Prof. Arvind Ayyer for his constant support and encouragement. The second author would like to thank Prof. Sivaramakrishnan Sivasubramanian for his constant support and encouragement. The authors sincerely thank Prof. Peter Cameron for some fruitful suggestions. The authors are also thankful to Prof. Angsuman Das for many helpful discussions during this work.
  The authors have been greatly inspired and motivated by participating in the Research Discussion on Graphs and Groups, organized by Department of Mathematics, CUSAT, India. 
  The first author was supported by NBHM Post Doctoral Fellowship grant 0204/52/2019/RD-II/339. The second author was supported by IIT Bombay Post Doctoral Fellowship. 
\bibliographystyle{amsplain}
\bibliography{gen-inv-lcp.bib}
\end{document}